\documentclass{article}
\usepackage{geometry}
\usepackage{graphics}
\usepackage{epstopdf}
\usepackage{color}
\usepackage{amsmath,amsfonts}
 \usepackage[pagewise]{lineno}
 \usepackage{tikz}
\usepackage{graphicx}
\usepackage{tikz-3dplot}
\usepackage{setspace}
\usepackage{fancyhdr}
\usepackage[ansinew]{inputenc}
\usepackage{comment}
\usepackage{amssymb}
\usepackage{latexsym}
\usepackage{graphics}
\usepackage{epstopdf}
\usepackage{xcolor}
\usepackage{fancybox}
\usepackage{amsmath,amsfonts}
\colorlet{Marino}{blue!50!black}
\colorlet{Naranja}{orange!80!black}
\colorlet{Vino}{red!50!green}
\usepackage{hyperref}
\begin{document}

\setlength{\parskip}{0.3\baselineskip}

\newtheorem{theorem}{Theorem}
\newtheorem{corollary}[theorem]{Corollary}
\newtheorem{lemma}[theorem]{Lemma}
\newtheorem{proposition}[theorem]{Proposition}
\newtheorem{definition}[theorem]{Definition}
\newtheorem{remark}[theorem]{Remark}
\renewcommand{\thefootnote}{\alph{footnote}}
\newenvironment{proof}{\smallskip \noindent{\bf Proof}: }{\hfill $\Box$\hspace{1in} \medskip \\ }

\newcommand{\sii}{\Leftrightarrow}
\newcommand{\imer}{\hookrightarrow}
\newcommand{\imerc}{\stackrel{c}{\hookrightarrow}}
\newcommand{\Con}{\longrightarrow}
\newcommand{\con}{\rightarrow}
\newcommand{\conf}{\rightharpoonup}
\newcommand{\confe}{\stackrel{*}{\rightharpoonup}}
\newcommand{\pbrack}[1]{\left( {#1} \right)}
\newcommand{\sbrack}[1]{\left[ {#1} \right]}
\newcommand{\key}[1]{\left\{ {#1} \right\}}
\newcommand{\dual}[2]{\langle{#1},{#2}\rangle}

\newcommand{\R}{{\mathbb R}}
\newcommand{\N}{{\mathbb N}}
\newcommand{\cred}[1]{\textcolor{red}{#1}}

\title{\bf Analyticity for Double Wall Carbon Nanotubes  Modeled as Timoshenko Beams with  Kelvin-Voigt and Intermediate Damping}
\author{Fredy Maglorio Sobrado  Su\'arez$^*$ \let\thefootnote\relax\footnote{ corresponding author:{\it e-mail:}{\rm fredy@utfpr.edu.br} (Fredy Maglorio Sobrado  Su\'arez)}\\
	Gilson Tumelero\\
Jackson Luchesi\\	
Marieli Musial Tumelero\\
  Santos Richard Wieller Sanguino Bejarano\\
{\small Department of  Mathematics, The  Federal University of Technological of Paran\'a, Brazil}
}

\date{ }

\maketitle
   \begin{center} `` In Memory of Susana and Maglorio"  \end{center}

\begin{abstract} 
	\noindent This  manuscript studies  a model of double-walled carbon nanotubes using two Timoshenko beams which are coupled by the  Van der Walls force $(y-u)$. Kelvin-Voigt type dampings  $(u_x-v)_{xt}$ and $(y_x-z)_{xt}$ and fractional dampings  $(-\partial_{xx})^\alpha v_t$ and $(-\partial_{xx})^\beta z_t$ in both beams have been considered. We show that our proposed model is well established and that the semigroup associated is exponentially stable and analytical for any $(\alpha, \beta) \in [0, 1]^2$. As a consequence of this, a result on the analyticity of a Timoshenko System is obtained.
\end{abstract} 

\bigskip
{\sc keyword:}   Stability, Exponential Decay,  Regularity, Analyticity,  DWCNTs System.

\setcounter{equation}{0}

\section{Introduction}
  With the advancement of nanotechnology, in recent decades, scientists have focused their efforts on discovering materials with nanomolecular structures. In this direction, an important discovery occurred in 1987, when the structures called carbon nanotubes (CNT) were discovered. These structures were presented to the scientific community in 1991 \cite{SL1991} as multi-walled carbon nanotubes (MWCNT).
 Carbon nanotubes are cylindrical macromolecules composed of carbon atoms in a periodic hexagonal matrix (see Figure 1) with $sp^2$ hybridization, similar to graphite \cite{HDTR08}. They are made as rolled sheets of graphene and can be as thick as a single carbon atom. They receive this name due to their tubular morphology in nanometric dimensions ($1nm=10^{-9}m$).
 \begin{figure}[ht]
 	\center
 	\begin{minipage}{0.4\textwidth}
 		\center
 		\includegraphics[width=0.4\textwidth]{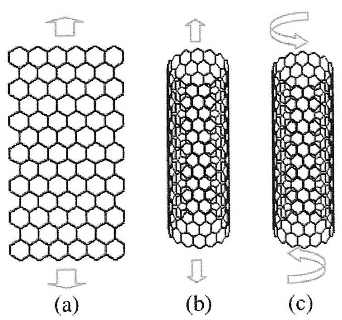}
 		\caption{Nanotube structure  \cite{DaSilva}} 
 		\label{centralizada}
 	\end{minipage}
 \end{figure}
 
 An analysis of the main properties of CNTs was presented in \cite{YM2003}, the study confirmed that CNTs have excellent properties mechanical, electrical, electronic and chemical. They are about ten times stronger and six times lighter than steel. They transmit electricity like a superconductor and are excellent temperature transmitters. Due to their electro-electronic and mechanical properties superior to those of the materials currently used, carbon nanotubes are already being used in products and equipment that require nanoscale structures.  
In \cite{SBW2011}, CNTs were classified in three ways: single-walled carbon nanotubes (SWCNT), double-walled carbon nanotubes (DWCNT) and  multi-walled carbon nanotubes (MWCNT). 

  The authors also noted that DWCNTs are an emerging class of carbon nanostructures and represent the simplest way to study the physical effects of coupling between the walls of carbon nanotubes.

In the future, CNTs should become the base material for nanoelectronics, nanodevices, and nanocomposites. The main problems that must be overcome for this to happen are the difficult controlled experiments at the nanoscale: the high cost of molecular dynamics simulations and the high time consumption of these simulations. In this direction, the Timoshenko beam model has been extensively used to understand continuum mechanics models of this material.
 Initially, CNTs were modeled using the Euler-Bernoulli beam model. However, this model ignores the effects of shear and rotation. According to \cite{YM2002, YM2003} the vibrations in carbon nanotubes are animated by high frequencies, higher than $1Thz$. In \cite {YM2005}, the authors showed that the effects of rotational inertia and shear are significant in the study of terahertz frequencies ($10^{12}$) becomming the Euler-Bernoulli beam model to CNTs questionable. Therefore, the Timoshenko-beam model should be used for terahertz vibrations of CNTs. For double-walled nanotubes DWCNT or multi-walled concentric nanotubes MWCNT, the most widely used continuum models in the literature assume that all nested MWCNT tubes remain coaxial during deformation and can therefore be described by a continuous model. However, this model cannot be used to describe the relative vibration between adjacent MWCNT tubes. In 2003, \cite{YM2003} proposed that concentric tube fittings be considered individual beams and that the deflections of all nested tubes be coupled by the Van der Waals interaction force between two adjacent tubes \cite{JC2007, JC2008}. Thus, each of the inner and outer tubes is modeled as a beam.

 In the pioneering work on the carbon nanotube model by Yoon et al. \cite{YM2004}, the authors proposed a coupled system of partial differential equations inspired by the Timoshenko beam model to model DWCNT. The model consists of the following equations:
                                
\begin{eqnarray*}
 \rho A_1\dfrac{\partial^2 Y_1}{\partial t^2}-\kappa G A_1\bigg(\dfrac{\partial^2 Y_1}{\partial x^2}-\dfrac{\partial \varphi_1}{\partial x}   \bigg)-P & = & 0,\\
 \rho I_1 \dfrac{\partial^2 \varphi_1}{\partial t^2}-E I_1 \dfrac{\partial^2 \varphi_1}{\partial x^2}-\kappa G A_1 \bigg( \dfrac{\partial Y_1}{\partial x}-\varphi_1 \bigg ) & = & 0,\\
 \rho A_2\dfrac{\partial^2 Y_2}{\partial t^2}-\kappa G A_2\bigg(\dfrac{\partial^2 Y_2}{\partial x^2}-\dfrac{\partial \varphi_2}{\partial x}   \bigg)+P & = & 0,\\
 \rho I_2 \dfrac{\partial^2 \varphi_2}{\partial t^2}-E I_2 \dfrac{\partial^2 \varphi_2}{\partial x^2}-\kappa G A_2 \bigg( \dfrac{\partial Y_2}{\partial x}-\varphi_2 \bigg ) & = & 0,
 \end{eqnarray*}
 where $Y_i$ and $\varphi_i$ ($i=1,2$) represent respectively the total deflection and the inclination due to the bending of the nanotube $i$ and the constants $I_i$, $A_i$ denote the moment of inertia and the cross-sectional area of the tube $i$,  respectively, and $P$ is the Van der Waals force acting on the interaction between the two tubes per unit of axial length. Also according to \cite{YM2004}, it can be seen that the deflections of the two tubes are coupled through the Van der Waals interaction $P$ (see \cite{Timoshenko1921}) between the two tubes, and as the tubes inside and outside of a DWCNTs are originally concentric, the Van der Waals interaction is determined by the spacing between the layers. Therefore, for a small-amplitude linear vibration, the interaction pressure at any point between the two tubes linearly depends on the difference in their deflection curves at that point, that is, it depends on the term
 \begin{equation}\label{Eq14ARamos}
 P=\jmath (Y_2-Y_1).
 \end{equation}

 In particular, the Van der Waals interaction coefficient $\jmath$ for the interaction pressure per unit axial length can be estimated based on an effective interaction width of the tubes as found in \cite{Ru2000, YM2003A}. Thus, this model treats each of the nested and concentric nanotubes as individual Timoshenko beams interacting in the presence of Van der Waals forces (see Figure: \eqref{DWCNTs}).
 \begin{figure}[ht]
\begin{minipage}{1\textwidth}
\center
\includegraphics[width=0.95\textwidth]{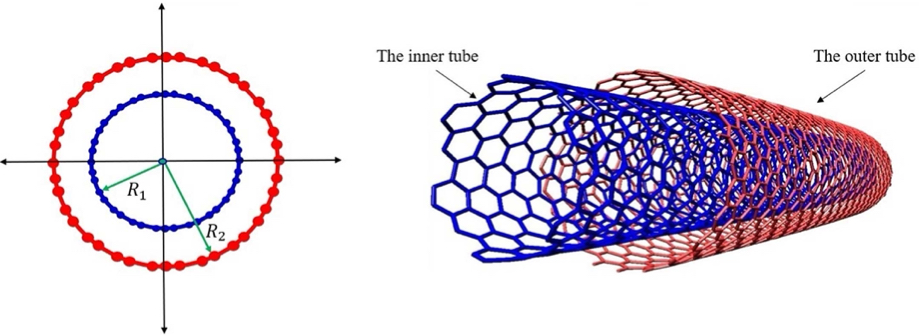}
\caption{2D and 3D Representations of the Double Wall Carbono Nanotubes Model \cite{Ramos2023CNTs}}
\label{DWCNTs}
\end{minipage}
\end{figure}
 
The study of asymptotic behavior and/or regularity for DWCNTs models, or for DWCNTs systems coupled to the heat equation governed by Fourier's law (DWCNTs-Fourier), are recent. In 2015, in the thesis \cite{Nunes2015}, the first results of the asymptotic behavior for DWCNTs models emerged, given by:
\begin{eqnarray}
\label{Eq0.1}
\rho_1 \varphi_{tt}-\kappa_1(\varphi_x-\psi)_x-\jmath(y-\varphi)+\alpha_0\varphi_t=0\quad{\rm in} \quad (0,l)\times (0,\infty),\\
\label{Eq0.2}
\rho_2\psi_{tt}-b_1\psi_{xx}-\kappa_1(\varphi_x-\psi)+\alpha_1\psi_t=0\quad{\rm in} \quad (0,l)\times (0,\infty),\\
\label{Eq0.3}
\rho_3y_{tt}-\kappa_2(y_x-z)_x+\jmath(y-\varphi)+\alpha_2y_t=0\quad{\rm in} \quad (0,l)\times (0,\infty),\\
\label{Eq0.4}
\rho_4z_{tt}-b_2z_{xx}-\kappa_2(y_x-z)+\alpha_3 z_t=0\quad{\rm in} \quad (0,l)\times (0,\infty),
\end{eqnarray}
with the initial conditions
\begin{eqnarray}
\label{Eq01.5}
\varphi(x,0)=\varphi_0(x),\quad \varphi_t(x,0)=\varphi_1(x),\quad \psi(x,0)=\psi_0(x),\quad \psi_t(x,0)=\phi_1(x) & {\rm in}\quad  (0,l),
\\
\label{Eq01.6}
y(x,0)=y_0(x),\quad y_t(x,0)=y_1(x),\quad z(x,0)=z_0(x),\quad z_t(x,0)=z_1(x) & {\rm in}\quad (0,l),
\end{eqnarray}
and subject to boundary conditions
\begin{eqnarray}
\label{Eq0.5}
\varphi(0,t)=\varphi(l,t)=\psi(0,t)=\psi(l,t)=0\quad{\rm for\quad all}\quad t>0,\\
\label{Eq0.6}
y(0,t)=y(l,t)=z(0,t)=z(l,t)=0\quad{\rm for\quad all}\quad t>0.
\end{eqnarray}

For the case that $\alpha_0=0$ and $\alpha_i>0$, for $i=1, 2, 3$, in \cite{Nunes2015} the author demonstrated the lack of exponential decay of the semigroup $(S(t))_{t\geq 0}$ associated with the system \eqref{Eq0.1}--\eqref{Eq0.6}, when $\frac{\rho_1}{\kappa_1}\not =\frac{\rho_2}{b_1}$ and $\jmath\big(\frac{\rho_2}{b_1}-\frac{\rho_1}{\kappa_1}\big)\not =\frac{\kappa_1}{ b_1}$. Furthermore, it has been proved that $(S(t))_{t\geq 0}$ is exponentially stable if $\chi= \frac{\kappa_1\rho_2-b_1\rho_1}{\kappa_1^2-\jmath \rho_2\kappa_1+\jmath b_1\rho_1}=0$ and $(S(t))_{t\geq 0}$ is polynomially stable with optimal rate $O(t^{-\frac{1}{2}})$ if $\chi\not=0$. In addition, in Chapter 4 of \cite{Nunes2015}, making use of the finite difference method, the author validates numerically the previously demonstrated results and  presents graphs of other dissipation cases.

 In 2023 \cite{MDCL2023}, the authors studied the one-dimensional equations for the double-wall carbon nanotubes modeled by coupled Timoshenko elastic beam system with nonlinear arbitrary localized damping:
 \begin{eqnarray}
\label{Eq0.20}
\rho_1 \varphi_{tt}-\kappa_1(\varphi_x-\psi)_x-\jmath(y-\varphi)+\alpha_1(x)g_1(\varphi_t)=0\quad{\rm in} \quad (0,l)\times (0,\infty),\\
\label{Eq0.21}
\rho_2\psi_{tt}-b_1\psi_{xx}-\kappa_1(\varphi_x-\psi)+\alpha_2(x)g_2(\psi_t)=0\quad{\rm in} \quad (0,l)\times (0,\infty),\\
\label{Eq0.22}
\rho_3y_{tt}-\kappa_2(y_x-z)_x+\jmath(y-\varphi)+\alpha_3(x)g_3(y_t)=0\quad{\rm in} \quad (0,l)\times (0,\infty),\\
\label{Eq0.23}
\rho_4z_{tt}-b_2z_{xx}-\kappa_2(y_x-z)+\alpha_4(x)g_4 (z_t)=0\quad{\rm in} \quad (0,l)\times (0,\infty),
\end{eqnarray}
where the localizing functions $\alpha_i(x)$ are supposed to be smooth and nonnegative, while the nonlinear functions $g_i(x), i = 1,\cdots, 4$, are continuous and monotonic increasing. The authors showed that the system \eqref{Eq0.20}--\eqref{Eq0.23} subject to Dirichlet boundary conditions with damping placed on an arbitrary small support, not quantized at the origin, leads to uniform (time asymptotic) decay rates for the energy function of the system.

In the same direction of this last paper, we would like to mention the work of Shubov and Rojas-Arenaza \cite{SR2010} where they considered the system \eqref{Eq0.20}-\eqref{Eq0.23}  with $\alpha_i(x)=1, g_i(s)=s, i=1,\cdots,4$,  subject to initial conditions \eqref{Eq01.5}-\eqref{Eq01.6} and boundary conditions:
\begin{equation}\label{Eq2.1MG}
\left\{\begin{array}{cc}
\kappa_1(\varphi_x-\psi)(l,t)=-\rho_2\gamma_1\varphi_t(l,t) & t\geq 0,\\
b_1\psi_x(l,t)=-\rho_2\gamma_2\psi_t(l,t),  & t\geq 0,\\
\kappa_2(y_x-z)(l,t)=\rho_4\gamma_3y_t(l,t), & t\geq 0,
\\
b_2z_x(l,t)=-\rho_4\gamma_4z_t(l,t), & t\geq 0.
\end{array}\right.
\end{equation}

Firstly, the authors proved  that the energy associated to the system,  with boundary conditions \eqref{Eq2.1MG}, is decreasing if $\jmath = 0$. Next, they proved that the semigroup generator is an unbounded non-self-adjoint operator with a compact resolvent.

In 2024 Guesmia \cite{AGuesmia24} studied the stability and well-posedness for double-walled carbon nanotubes modeled as two one-dimensional linear Timoshenko beams coupled in a bounded domain under friction damping. In this work the author studies good approach by applying the theory of semigroups of linear operators. Furthermore, it has been proven several strong non-exponential, exponential, polynomial and non-polynomial stability results depending on the number of friction dampers, their position and some connections between the coefficients. In some cases, optimization of the polynomial decay rate is also demonstrated. The tests of these stability results are based on a combination of the energy method and the frequency domain approach.

Recently, 2023/2024, two new studies also emerged for the DWCNTs-Fourier system: One of them is the system studied in \cite{Ramos2023CNTs}, in 2023. The system is given by:
\begin{eqnarray}
\label{Eq0.10}
\rho_1 \varphi_{tt}-\kappa_1(\varphi_x-\psi)_x-\jmath(y-\varphi)+\gamma_1\varphi_t=0\quad{\rm in} \quad (0,l)\times (0,\infty),\\
\label{Eq0.11}
\rho_2\psi_{tt}-b_1\psi_{xx}-\kappa_1(\varphi_x-\psi)+\delta\theta_{xx}=0\quad{\rm in} \quad (0,l)\times (0,\infty),\\
\label{Eq0.12}
\rho_3y_{tt}-\kappa_2(y_x-z)_x+\jmath(y-\varphi)+\gamma_2y_t=0\quad{\rm in} \quad (0,l)\times (0,\infty),\\
\label{Eq0.13}
\rho_4z_{tt}-b_2z_{xx}-\kappa_2(y_x-z)+\gamma_3 z_t=0\quad{\rm in} \quad (0,l)\times (0,\infty),\\
\label{Eq0.14}
\rho_5\theta_t-K\theta_{xx}+\beta\psi_t=0\quad {\rm in}\quad (0,l)\times (0,\infty),
\end{eqnarray}
subject to boundary conditions 
\begin{eqnarray}
	\label{Eq0.5a}
	\varphi(0,t)=\varphi(l,t)=\psi(0,t)=\psi(l,t)=0\quad{\rm for\quad all}\quad t>0,\\
	\label{Eq0.6a}
	y(0,t)=y(l,t)=z(0,t)=z(l,t)=0\quad{\rm for\quad all}\quad t>0,\\
\label{Eq0.15}
\theta(0,t)=\theta(l,t)=0\quad{\rm for\quad all}\quad t>0.
\end{eqnarray}
\quad\, Note that the system \eqref{Eq0.10}--\eqref{Eq0.15} presents three friction dissipators (weak damping): $\gamma_1\varphi_ t,\gamma_2y_t$ and $\gamma_3 z_t$. The authors applied semigroup theory of linear operators to demonstrate the exponential stabilization of the semigroup $S(t)$ associated with the system \eqref{Eq0.10}--\eqref{Eq0.15}, and their results are independent of the relationship between the coefficients. Furthermore, they analyzed the totally discrete problem using a finite difference scheme, introduced by a space-time discretization that combines explicit and implicit integration methods. The authors also show the construction of numerical energy and simulations that validate the theoretical results of exponential decay and convergence rates.

In 2024, Su\'arez et al. \cite{Suarez2024}, studied two DWCNT-Fourier systems provided with fractional damping. For the first system the Exponential Decay (Stability) is demonstrated and for the second system they approched the stability and regularity. The first system is a generalization of the model presented in \cite{Ramos2023CNTs}. In this system the authors consider three fractional damping given by: $\gamma_1(-\partial_{xx})^{\tau_1}\varphi_t$, $\gamma_2(- \partial_{xx})^{\tau_2}y_t$ and $\gamma_3(-\partial_{xx})^{\tau_3}z_t$, for the parameters $\tau_i, i=1,2,3$, varying in the interval $[0,1]$. Note that when $(\tau_1,\tau_2,\tau_3)=(0,0,0)$ the system is the one studied in \cite{Ramos2023CNTs}. The second system studied in \cite{Suarez2024} is given by:
\begin{eqnarray*}
\label{Eq2.1}
\rho_1 \varphi_{tt}-\kappa_1(\varphi_x-\psi)_x-\jmath(y-\varphi)+\gamma_1(-\partial_{xx})^{\beta_1}\varphi_t=0\quad{\rm in} \quad (0,l)\times (0,\infty),\\
\label{Eq2.2}
\rho_2\psi_{tt}-b_1\psi_{xx}-\kappa_1(\varphi_x-\psi)+\delta \theta_{xx}=0\quad{\rm in} \quad (0,l)\times (0,\infty),\\
\label{Eq2.3}
\rho_3y_{tt}-\kappa_2(y_x-z)_x+\jmath(y-\varphi)+\gamma_2(-\partial_{xx})^{\beta_2}y_t=0\quad{\rm in} \quad (0,l)\times (0,\infty),\\
\label{Eq2.4}
\rho_4z_{tt}-b_2z_{xx}-\kappa_2(y_x-z)+\gamma_3(-\partial_{xx})^{\beta_3}z_t=0\quad{\rm in} \quad (0,l)\times (0,\infty),\\
\label{Eq2.5}
\rho_5\theta_t-K\theta_{xx}-\delta \psi_{xxt}=0\quad{\rm in} \quad (0,l)\times (0,\infty),
\end{eqnarray*}
subject to boundary conditions
\begin{eqnarray*}
\label{Eq2.7}
\varphi(0,t)=\varphi(l,t)=\psi(0,t)=\psi(l,t)=0\quad{\rm for\quad all}\quad t>0,\\
\label{Eq2.8}
y(0,t)=y(l,t)=z(0,t)=z(l,t)=0\quad{\rm for\quad all}\quad t>0,\\
\label{Eq2.9}
\theta(0,t)=\theta(l,t)=0\quad{\rm for\quad all}\quad t>0,
\end{eqnarray*}
and initial conditions
\begin{eqnarray*}
\label{Eq2.10}
\varphi(x,0)=\varphi_0(x),\; \varphi_t(x,0)=\varphi_1(x),\; \psi(x,0)=\psi_0(x),\quad {\rm for}\;x\in(0,l),\\
\label{Eq2.11}
\psi_t(x,0)=\psi_1(x),\; y(x,0)=y_0(x),\; y_t(x,0)=y_1(x), \quad {\rm for}\;x\in(0,l),\\
\label{Eq2.12}
z(x,0)=z_0(x),\; z_t(x,0)=z_1(x),\; \theta(x,0)=\theta_0(x),\quad{\rm for}\;x\in(0,l).
\end{eqnarray*}

In the recent paper published in 2025 \cite{Apalara2025}, the authors studied  the well-posedness and stability of DWCNTs, modeled as a double-Timoshenko system coupled with Lord-Shulman thermoelasticity. They proved the existence of a unique solution that is exponentially stable irrespective of any assumptions on the coefficient of the system:
\begin{eqnarray*}
	\label{Eq2.1a}
	\rho_1 \varphi_{tt}-\kappa_1(\varphi_x-\psi)_x-\jmath(y-\varphi)+\gamma_1\varphi_t=0\quad{\rm in} \quad (0,l)\times (0,\infty),\\
	\label{Eq2.2b}
	\lambda_1\psi_{tt}-b_1\psi_{xx}-\kappa_1(\varphi_x-\psi)+\delta( \theta_{x}+\tau\theta_{tx})=0\quad{\rm in} \quad (0,l)\times (0,\infty),\\
	\label{Eq2.3c}
	\rho_2y_{tt}-\kappa_2(y_x-z)_x+\jmath(y-\varphi)+\gamma_2y_t=0\quad{\rm in} \quad (0,l)\times (0,\infty),\\
	\label{Eq2.4d}
	\lambda_2z_{tt}-b_2z_{xx}-\kappa_2(y_x-z)+\gamma_3z_t=0\quad{\rm in} \quad (0,l)\times (0,\infty),\\
	\label{Eq2.5e}
	\rho_3(\theta_{t}+\tau\theta_{tt})-K\theta_{xx}-\beta \psi_{xt}=0\quad{\rm in} \quad (0,l)\times (0,\infty),
\end{eqnarray*}
subject to initial and Dirichlet boundary conditions:
\begin{eqnarray*}
	\label{Eq2.7f}
	\varphi(0,t)=\varphi(l,t)=\psi(0,t)=\psi(l,t)=0\quad{\rm for\quad all}\quad t\geq0,\\
	\label{Eq2.8g}
	y(0,t)=y(l,t)=z(0,t)=z(l,t)=0\quad{\rm for\quad all}\quad t\geq0,\\
	\label{Eq2.9h}
	\theta_x(0,t)=\theta_x(l,t)=0\quad{\rm for\quad all}\quad t\geq0,\\
	\label{Eq2.10i}
	\varphi(x,0)=\varphi_0(x),\; \varphi_t(x,0)=\varphi_1(x),\; \psi(x,0)=\psi_0(x),\; \psi_t(x,0)=\psi_1(x),\quad {\rm for}\;x\in(0,l),\\
	\label{Eq2.11j}
	\; y(x,0)=y_0(x),\; y_t(x,0)=y_1(x),\; z(x,0)=z_0(x),\; z_t(x,0)=z_1(x), \quad {\rm for}\;x\in(0,l),\\
	\label{Eq2.12k}
	 \theta(x,0)=\theta_0(x,t),\;  \theta_t(x,0)=\theta_1(x,t),\quad{\rm for}\;x\in(0,l).
\end{eqnarray*}

Since the mathematical model of DWCNTs studied here is given by the coupling of two Timo-shenko beams through the Van de Waals interaction force, it is natural to think that previous studies of stability and regularity of Timoshenko beams lead us to obtain results of stability and regularity of the DWCNTs.  We will cite some of these works below.

In 2005,  Raposo et al.\cite{Raposo2005} studied the Timoshenko system, provided with two frictional dissipations $\varphi_t$ and $\psi_t$, and proved that the semigroup associated with the system decays exponentially. In 2016, for the same Timoshenko system, when the stress-strain constitutive law is of Kelvin-Voigt type, given by
\begin{equation}\label{Dtensores}
S=\kappa(\varphi_x+\psi)+\gamma_1(\varphi_x+\psi)_t\qquad\text{and}\qquad M=b\psi_x+\psi_{xt}, 
\end{equation} 
Malacarne and Rivera in \cite{AMJR2016} showed that $S(t)$ is analytical if and only if the viscoelastic damping is present in both the shear stress and the bending moment. Otherwise, the corresponding semigroup is not exponentially stable, independent of the choice of coefficients. They also showed that the solution decays polynomially to zero as $t^{-1/2}$, no matter where the viscoelastic mechanism is effective and that the rate is optimal whenever the initial data are taken on the domain of the infinitesimal operator. In 2023, Su\'arez \cite{Suarez2023} studied the regularity of the model given in \cite{Raposo2005}, replacing the two damping weaks $\varphi_t$ and $\psi_t$  with fractional dampings $(-\partial_{xx})^\tau\varphi_t$ and $(-\partial_{xx})^\sigma\psi_t$, where the parameters $\tau$ and $\sigma$ belong to interval $[0,1]$, and proved the existence of Gevrey classes $s>\frac{r+1}{2r}$ with $r=\min\{\tau,\sigma\}$, for all $\tau, \sigma \in (0,1)$, of the semigroup $S(t)$ associated to the system, and analyticity of $S(t)$ when the two parameters $\tau$ and $\sigma$ vary in the interval $[1/2,1]$.

For some references on stability  and regularity of coupled systems, we recommend consulting \cite{AmmariShelTebou2022,MH2019,SCRT1990,DOroPata2016,Tebou2020,KLiuH2021,KLiuTebou2022,LT1998A,LiuR95,FLF2023,HC2019,HSLiuRacke2019,LTebou2013,Tebou2021}

Motivated by the recent works of Malacarne-Rivera, Su\'arez and Guesmia: \cite{AMJR2016, Suarez2023, AGuesmia24} and considering similar dissipations as given in \eqref{Dtensores}, we study the system given by:
\begin{eqnarray}
\label{IEq1.1}
\rho_1 u_{tt}-\kappa_1(u_x-v)_x-m(y-u)-\gamma_1(u_x-v)_{xt}=0\quad{\rm in} \quad (0,l)\times (0,\infty),\\
\label{IEq1.2}
\rho_2v_{tt}-b_1v_{xx}-\kappa_1(u_x-v)-\gamma_1(u_x-v)_t+\gamma_2(-\partial_{xx})^\alpha v_t=0\quad{\rm in} \quad (0,l)\times (0,\infty),\\
\label{IEq1.3}
\rho_3y_{tt}-\kappa_2(y_x-z)_x+m(y-u)-\gamma_3(y_x-z)_{xt}=0\quad{\rm in} \quad (0,l)\times (0,\infty),\\
\label{IEq1.4}
\rho_4z_{tt}-b_2z_{xx}-\kappa_2(y_x-z)-\gamma_3(y_x-z)_t+\gamma_4(-\partial_{xx})^\beta z_t=0\quad{\rm in} \quad (0,l)\times (0,\infty),
\end{eqnarray}
where $m, b_1, b_2,\kappa_1, \kappa_2,  \rho_1,\rho_2,\rho_3,\rho_4, \gamma_1,\gamma_2,\gamma_3,\gamma_4$ are real positive numbers, $u_0,u_1,v_0,v_1, y_0,\\ y_1, z_0, z_1$ are given functions and $\alpha$ and $\beta$ belong to $[0,1]$. Subject to boundary conditions 
\begin{eqnarray}
\label{IEq1.5}
\quad u(0,t)=u(l,t)=v_x(0,t)=v_x(l,t)=0,\quad\quad t>0,\\
\label{IEq1.6}
y(0,t)=y(l,t)=z_x(0,t)=z_x(l,t)=0,\quad\quad t>0,
\end{eqnarray}
and initial data
\begin{eqnarray}
\label{IEq1.7}
u(x,0)=u_0(x),\; u_t(x,0)=u_1(x),\; v(x,0)=v_0(x),\quad \quad x\in(0,l),\\
\label{IEq1.8}
v_t(x,0)=v_1(x),\; y(x,0)=y_0(x),\; y_t(x,0)=y_1(x), \quad\quad x\in(0,l),\\
\label{IEq1.9}
z(x,0)=z_0(x),\;z_t(x,0)=z_1(x),\quad\quad x\in(0,l).
\end{eqnarray}

The main contribution of this article is to show that the corresponding semigroup $S(t)$  is analytic for all $(\alpha,\beta)\in [0, 1]^2=[0,1]\times[0,1]$. Many previous works on Timoshenko beams or DWCNTs obtain analyticity only in
	restricted subregions of the parameters (for example, $\alpha,\beta\geq \frac{1}{2}$
	as in the work by Suárez
	in 2023 \cite{Suarez2023} or $\alpha=1$ as in \cite{JAIB2025,AMJR2016}).
	Here, it is demonstrated that the entire region of the square $[0,1]^2$ is covered.
	This means that, regardless of the intensity of the fractional dissipation (from strong
	classical Kelvin-Voigt dissipations to cases of weak frictional dissipation), the system
	maintains analytic behavior.
	
In addition, as a consequence of this result, we obtain analyticity for all $\alpha\in[0,1]$ for the following Timoshenko System:
\begin{eqnarray*}
	\label{TTEq1.1}
	\rho_1 u_{tt}-\kappa_1(u_x-v)_x-\gamma_1(u_x-v)_{xt}=0\quad{\rm in} \quad (0,l)\times (0,\infty),\\
	\label{TTEq1.2}
	\rho_2v_{tt}-b_1v_{xx}-\kappa_1(u_x-v)-\gamma_1(u_x-v)_t+\gamma_2(-\partial_{xx})^\alpha    v_t=0\quad{\rm in} \quad (0,l)\times (0,\infty).
\end{eqnarray*}
subject to boundary conditions 
\begin{equation*}
	\label{TTq1.3}
	\quad u(0,t)=u(l,t)=v_x(0,t)=v_x(l,t)=0,\quad\quad t>0,
\end{equation*}
and initial data
\begin{equation*}
	\label{TTEq1.4}
	u(x,0)=u_0(x),\; u_t(x,0)=u_1(x),\; v(x,0)=v_0(x),\; v_t(x,0)=v_1(x),\quad \quad x\in(0,l).
\end{equation*}

{\bf Why is analyticity in comprehensive coverage of the parameter region relevant?}
\begin{itemize}
\item {\bf Strengthening of Solution Regularity:}
Exponential stability already guarantees that the energy decays uniformly over time.
Analyticity, however, is a much stronger result: it implies that the temporal evolution
of the system is infinitely differentiable and that the solutions belong to higher regularity
classes (including Gevrey classes in particular cases).
This is essential for control problems, numerical approximation, and spectral theory.
\item {\bf Implications for Nanotechnology:} 
In physical models of nanotubes, $\alpha$ and $\beta$ represent different dissipation regimes associated with microscopic mechanisms of memory and internal damping.
To prove that the model is analytic across the entire range of
these parameters increases its practical applicability and reliability in multiscale
simulations. In particular, analyticity ensures that the system reacts smoothly to external perturbations, which is crucial for applications in nanoelectronics and resonant devices on the THz scale.
\end{itemize}
This advance positions the model among the most robust ever obtained for DWCNTs (and Timoshenko Systems), overcoming limitations of previous works and paving the way for studies on optimal
control, stable numerical simulation, and thermoelastic or delayed extensions.

This manuscript is organized as follows. Section 1 is Introduction. Section 2 contains notations and auxiliary facts. In Section 3 we study the well-posedness making use of semigroup theory. Section 4 is dedicated to the stability of $S(t)$ and we show that the system \eqref{IEq1.1}-\eqref{IEq1.9} decays exponentially when both of the parameters $\alpha, \beta$ vary in the interval $[0,1]$, this study is also approached using the theory of linear semigroups. In Section 5, we prove the analyticity of the semigroup $S(t)$ associated with the system \eqref{IEq1.1}-\eqref{IEq1.9} for the parameters $(\alpha,\beta)$ in the square $[0,1]^ 2$ using semigroup theory. Section 6 contains analyticity results for Timoshenko systems as a particular case. Finally, Section 7 is a conclusion.

\section{Notations and Auxiliary Facts }
Let $l>0$ be a finite number. We denote the standard Hilbert space $L^2(0, l)$ with inner product and norm:
\begin{equation*}
	\dual{u}{v}=\int_{0}^{l}u(x)\overline{v(x)} dx,\,\,\,\,\, \|u\|=\Big(\int_{0}^{l}|u(x)|^2dx\Big)^{\frac{1}{2}}.
\end{equation*}
The closed subspace of functions with null mean of the $L^2(0,l)$ is denoted by
\begin{equation*}
	L_*^2(0,l)=\bigg\{ u\in L^2(0,l)\colon \int_0^l u(x)dx=0\bigg\}.
\end{equation*}
It is a Hilbert space with the $L^2-$norm. The space $H_0^1(0, l)$ stands for the usual Sobolev space and $H_*^1(0,l) = H^1(0,L)\cap L_*^2(0, l)$ equipped with the norms $\|u\|_{H_0^1(0, l)}=\|u\|_{H_*^1(0, l)}=\|u_x\|$.

The operators
\begin{equation}\label{OAAaste}
\left\{\begin{array}{l}
A=-\partial_{xx}\colon D(A)\subset L^2(0,l)\to L^2(0,l),\\
A_*=-\partial_{xx}\colon D(A_*)\subset L^2_*(0,l)\to L^2_*(0,l),
\end{array}\right.
\end{equation}
defined in the subspaces
\begin{equation}\label{DOAAaste}
\left\{\begin{array}{l}
 D(A)=H^2(0,l)\cap H_0^1(0,l),\\
D(A_*)=\{ v\in H^2(0,l)\cap L_*^2(0,l)\colon v_x(0)=v_x(l)=0\},

\end{array}\right.
\end{equation}
are positive seft-adjoint and have a compact inverse. Therefore, the operators $A^\sigma, A_*^\sigma$ are bounded for $\sigma\leq 0$ and they are positive seft-adjoint for any $\sigma\in\mathbb{R}$. Moreover, the embeddings
\begin{equation*}
\mathfrak{D}(A^{\sigma_1})\hookrightarrow \mathfrak{D}(A^{\sigma_2})\quad{\rm and}\quad \mathfrak{D}(A_*^{\sigma_1})\hookrightarrow \mathfrak{D}(A_*^{\sigma_2})
\end{equation*}
are continuous for $\sigma_1>\sigma_2$. Here, the norms in $D(A^\sigma)$ and $D(A_*^\sigma)$, for  $\sigma\geq 0$, are given by $\|u\|_{D(A^\sigma)}:=\|A^\sigma u\|$ and $\|v\|_{D(A_*^\sigma)}:=\|A^\sigma_* v \|$ respectively.

With the considerations presented above, the spectrum of  the operators $A$ and $A_*$ are constituted only by positive eigenvalues. The eigenvalues for both operators are given by $\sigma_n^2$ where $\sigma_n=\frac{n\pi}{l},\, n\in\mathbb{N}$, and the corresponding unitary eigenfunctions associated to these eigenvalues are
\begin{equation}\label{Eq08HC}
e_n(x)=\sqrt{\dfrac{2}{l}}\sin(\sigma_nx)\qquad\text{and}\qquad e^*_n(x)=\sqrt{\dfrac{2}{l}}\cos(\sigma_nx).
\end{equation}
The sequences $\{e_n\}_{n\in\mathbb{N}}$ and  $\{e^*_n\}_{n\in\mathbb{N}}$ constitute a Hilbert's base for the spaces $L^2(0,l)=D(A^0)$ and  $L_*^2(0,l)=D(A_*^0)$ respectively, thus for $u\in D(A^0)$ and $v\in D(A_*^0)$, 
\begin{equation*}
u=\sum\limits_{n=1}^\infty \dual{u}{e_n}e_n\qquad\text{and}\qquad v=\sum\limits_{n=1}^\infty \dual{v}{e_n^*}e_n^*.
\end{equation*}
Note that, for $u\in D(A^{\sigma+1/2})$, we have
\begin{equation*}
A^{\sigma+1/2}u=\sum\limits_{n=1}^\infty \sigma_n^{2\sigma+1}\dual{u}{e_n}e_n\qquad \text{and} \qquad
A_*^\sigma\partial_xu=\sum\limits_{n=1}^\infty \sigma_n^{2\sigma+1}\dual{u}{e_n}e_n^*,
\end{equation*}
from where it following, by Parseval's identity, that
\begin{equation}
\label{Eq09HC}
\|A^{\sigma+1/2}u\|=\|A^\sigma_* \partial_xu\|.
\end{equation}
In particular, $\|A^\frac{1}{2}u\|=\|\partial_xu\|$ for $\sigma=0$. Similarly, for $v\in D(A_*^{\sigma+1/2})$, we find 
\begin{equation}\label{E10HC}
\|A_*^{\sigma+1/2}v\|=\|A^\sigma \partial_xv\|,
\end{equation}
and for $u\in D(A^{\sigma_0}), v\in D(A_*^{\sigma_0})$, with $\sigma_0=\max\{\sigma,1/2\}$, we have
\begin{equation}\label{Eq11HC}
\dual{A_*^\sigma v}{\partial_xu}=-\dual{\partial_xv}{A^\sigma u}.
\end{equation}
\begin{theorem}[Hille-Yosida, see \cite{Pazy}, p. 8]\label{THY} Let $\mathcal{H}$ be a Banach space. A linear (unbounded) operator $\mathbb{B}$ is the infinitesimal generator of a $C_0-$semigroup of contractions $S(t)$, $ t\geq 0$, if and only if\\
$(i)$ $\mathbb{B}$ is closed and $\overline{\mathfrak{D}(\mathbb{B})}=\mathcal{H}$,\\
$(ii)$ The resolvent set $\rho(\mathbb{B})$ of $\mathbb{B}$ contains $\mathbb{R}^+$ and for every $\lambda>0$, 
\begin{equation*}
\|(\lambda I-\mathbb{B})^{-1}\|_{\mathcal{L}(\mathcal{H})}\leq\dfrac{1}{\lambda}.
\end{equation*}
\end{theorem}
\begin{theorem}[Lions' Interpolation, see \cite{EN2000}, Theorem  5.34 ]\label{Lions-Landau-Kolmogorov}  Let $\alpha<\beta<\gamma$. The there exists a constant $L=L(\alpha,\beta,\gamma)$ such that
\begin{equation}\label{ILLK}
\|A^\beta u\|\leq L\|A^\alpha u\|^\frac{\gamma-\beta}{\gamma-\alpha}\cdot \|A^\gamma u\|^\frac{\beta-\alpha}{\gamma-\alpha}
\end{equation}
for every $u\in\mathfrak{D}(A^\gamma)$.
\end{theorem}
\begin{theorem}[see \cite{LiuZ},  Theorem 1.2.4] \label{TLiuZ}
	Let $\mathbb{B}$ be a linear operator with domain $\mathfrak{D}(\mathbb{B})$ dense in a Hilbert space $\mathbb{H}$. If $\mathbb{B}$ is dissipative and $0\in\rho(\mathbb{B})$, the resolvent set of $\mathbb{B}$, then $\mathbb{B}$ is the generator of a $C_0$-semigroup of contractions on $\mathbb{H}$.
\end{theorem}
As a consequence of the previous Theorem \ref{TLiuZ},  we have
\begin{theorem}
	Given $U_0\in\mathbb{H}$ there exists a unique weak solution $U$ to  the problem \eqref{Fabstrata} satisfying 
	$$U\in C([0, +\infty), \mathbb{H}).$$
	Furthermore,  if $U_0\in  \mathfrak{D}(\mathbb{B}^k), \; k\in\mathbb{N}$, then the solution $U$ of \eqref{Fabstrata} satisfies
	$$U\in \bigcap_{j=0}^kC^{k-j}([0,+\infty),  \mathfrak{D}(\mathbb{B}^j).$$
	
\end{theorem}
\section{Well-posedness}
Let $\alpha,\beta$ be real positive numbers. Consider the following linear system:  
\begin{eqnarray}
	\label{Eq1.1}
	\rho_1 u_{tt}-\kappa_1(u_x-v)_x-m(y-u)-\gamma_1(u_x-v)_{xt}=0\quad{\rm in} \quad (0,l)\times (0,\infty),\\
	\label{Eq1.2}
	\rho_2v_{tt}-b_1v_{xx}-\kappa_1(u_x-v)-\gamma_1(u_x-v)_t+\gamma_2(-\partial_{xx})^\alpha v_t=0\quad{\rm in} \quad (0,l)\times (0,\infty),\\
	\label{Eq1.3}
	\rho_3y_{tt}-\kappa_2(y_x-z)_x+m(y-u)-\gamma_3(y_x-z)_{xt}=0\quad{\rm in} \quad (0,l)\times (0,\infty),\\
	\label{Eq1.4}
	\rho_4z_{tt}-b_2z_{xx}-\kappa_2(y_x-z)-\gamma_3(y_x-z)_t+\gamma_4(-\partial_{xx})^\beta z_t=0\quad{\rm in} \quad (0,l)\times (0,\infty),
\end{eqnarray}
subject to boundary conditions 
\begin{eqnarray}
	\label{Eq1.5}
	\quad u(0,t)=u(l,t)=v_x(0,t)=v_x(l,t)=0,\quad\quad t>0,\\
	\label{Eq1.6}
	y(0,t)=y(l,t)=z_x(0,t)=z_x(l,t)=0,\quad\quad t>0,
\end{eqnarray}
and initial data
\begin{eqnarray}
	\label{Eq1.7}
	u(x,0)=u_0(x),\; u_t(x,0)=u_1(x),\; v(x,0)=v_0(x),\quad \quad x\in(0,l),\\
	\label{Eq1.8}
	v_t(x,0)=v_1(x),\; y(x,0)=y_0(x),\; y_t(x,0)=y_1(x), \quad\quad x\in(0,l),\\
	\label{Eq1.9}
	z(x,0)=z_0(x),\;z_t(x,0)=z_1(x),\quad\quad x\in(0,l),
\end{eqnarray}
where $m, b_1, b_2,\kappa_1, \kappa_2,  \rho_1,\rho_2,\rho_3,\rho_4, \gamma_1,\gamma_2,\gamma_3,\gamma_4$ are real positive numbers and $u_0,u_1,v_0,v_1, y_0,\\ y_1, z_0, z_1$ are given functions. 

Using  the operator $A_*$ defined in \eqref{OAAaste}, the system \eqref{Eq1.1}-\eqref{Eq1.4} can be rewritten as
\begin{eqnarray}
\label{Eq1.13}
\rho_1 u_{tt}-\kappa_1(u_x-v)_x-m(y-u)-\gamma_1(u_x-v)_{xt}=0\quad{\rm in} \quad (0,l)\times (0,\infty),\\
\label{Eq1.14}
\rho_2v_{tt}+b_1A_*v-\kappa_1(u_x-v)-\gamma_1(u_x-v)_t+\gamma_2A_*^\alpha v_t=0\quad{\rm in} \quad (0,l)\times (0,\infty),\\
\label{Eq1.15}
\rho_3y_{tt}-\kappa_2(y_x-z)_x+m(y-u)-\gamma_3(y_x-z)_{xt}=0\quad{\rm in} \quad (0,l)\times (0,\infty),\\
\label{Eq1.16}
\rho_4z_{tt}+b_2A_*z-\kappa_2(y_x-z)-\gamma_3(y_x-z)_t+\gamma_4 A_*^{\beta}z_t=0\quad{\rm in} \quad (0,l)\times (0,\infty).
\end{eqnarray}

For every solution of the system \eqref{Eq1.13}-\eqref{Eq1.16} the total energy $\mathfrak{E}\colon \mathbb{R}^+_0\to\mathbb{R}^+$ is given by    
\begin{multline}\label{Energia01}
\mathfrak{E}(t)=\frac{1}{2}\bigg[ \rho_1 \|u_t\|^2+\rho_2\|v_t\|^2+\rho_3\|y_t\|^2+  \rho_4\|z_t\|^2+ \kappa_1\|u_x-v\|^2+\kappa_2\|y_x-z\|^2+m\|y-u\|^2 \\+b_1\|A_*^{1/2}v\|^2 +b_2 \|A_*^{1/2}z\|^2 \bigg ](t), \quad t\geq 0.
\end{multline}
Then, a straightforward computation gives
\begin{equation}\label{Dissipa01}
\dfrac{d}{dt}\mathfrak{E}(t)=-\gamma_1\|(u_x-v)_t\|^2 -\gamma_2\|A_*^{\alpha/2} v_t\|^2 - \gamma_3\|(y_x-z)_t\|^2 -\gamma_4 \|A_*^{\beta/2}z_t\|^2 \leq 0,
\end{equation}
from where it follows that $\mathfrak{E}$ is non-increasing with $\mathfrak{E}(t)\leq \mathfrak{E}(0)$ for all $t\geq 0$.

We introduce the phase space over the field $\mathbb{C}$ of complex numbers
\begin{eqnarray*}
    \label{EspacoFase}
    \mathcal{H}:=\bigg[D(A^\frac{1}{2})\times D(A^0) \times D(A_*^\frac{1}{2})\times D(A_*^0) \bigg]^2.
\end{eqnarray*}
It is a Hilbert space equipped with the inner product
\begin{eqnarray*}
	\dual{ U_1}{U_2}_{\mathcal{H}} & := & \rho_1\dual{\tilde{u_1}}{\tilde{u_2}}+\rho_2\dual{\tilde{v_1}}{\tilde{v_2}}+\rho_3\dual{\tilde{y_1}}{\tilde{y_2}}+\rho_4\dual{\tilde{z_1}}{\tilde{z_2}}+\kappa_1\dual{u_{1x}-v_1}{u_{2x}-v_2}\\
	&  & + \kappa_2\dual{y_{1x}-z_1}{y_{2x}-z_2} +m\dual{y_1-u_1}{y_2-u_2} \\
	& &+b_1\dual{A_{*}^{1/2}v_{1}}{A_{*}^{1/2}v_{2}}+b_2\dual{A_{*}^{1/2}z_{1}}{A_{*}^{1/2}z_{2}},
\end{eqnarray*}
for $U_i=(u_i,\tilde{u_i},v_i, \tilde{v_i},y_i, \tilde{y_i},z_i,\tilde{z_i})^T\in \mathcal{H}$,  $i=1,2$ and induced norm
\begin{multline}\label{NORM}
	\|U\|_{\mathcal{H}}^2:=\rho_1 \|\tilde{u}\|^2+\rho_2\|\tilde{v}\|^2+\rho_3\|\tilde{y}\|^2+  \rho_4\|\tilde{z}\|^2+ \kappa_1\|u_x-v\|^2+\kappa_2\|y_x-z\|^2+m\|y-u\|^2 \\+b_1\|A_*^{1/2}v\|^2 +b_2 \|A_*^{1/2}z\|^2,
\end{multline}
   where $u_t=\tilde{u}$, $v_t=\tilde{v}$, $y_t=\tilde{y}$ and $z_t=\tilde{z}$.
 The initial-boundary value problem \eqref{Eq1.7}-\eqref{Eq1.16} can be reduced to the following abstract initial value problem:
 \begin{equation}\label{Fabstrata}
    \frac{d}{dt}U(t)=\mathbb{B} U(t),\quad    U(0)=U_0,
\end{equation}
 where $U(t)=(u,\tilde{u},v, \tilde{v},y, \tilde{y},z, \tilde{z})^T$,  $U_0=(u_0,u_1,v_0,v_1, y_0,y_1,z_0,z_1)^T$  and the operator $\mathbb{B}$ 
 is defined by
\begin{equation}\label{operadorB}
 \mathbb{B}U:=\left(\begin{array}{c}
 \tilde{u}\\
 -\frac{\kappa_1}{\rho_1} Au -\frac{\kappa_1}{\rho_1} v_x +\frac{m}{\rho_1} y-\frac{m}{\rho_1} u -\frac{\gamma_1}{\rho_1} A\tilde{u}-\frac{\gamma_1}{\rho_1} \tilde{v}_x\\
 \tilde{v}\\
-\frac{b_1}{\rho_2} A_*v+\frac{\kappa_1}{\rho_2}u_x-\frac{\kappa_1}{\rho_2} v +\frac{\gamma_1}{\rho_2}\tilde{u}_x-\frac{\gamma_1}{\rho_2} \tilde{v}-\frac{\gamma_2}{\rho_2}A_*^\alpha \tilde{v}\\
\tilde{y}\\
 -\frac{\kappa_2}{\rho_3} Ay-\frac{\kappa_2}{\rho_3} z_x-\frac{m}{\rho_3}y+\frac{m}{\rho_3}u-\frac{\gamma_3}{\rho_3} A\tilde{y}-\frac{\gamma_3}{\rho_3} \tilde{z}_x\\
\tilde{z}\\
-\frac{b_2}{\rho_4}A_*z+\frac{\kappa_2}{\rho_4} y_x-\frac{\kappa_2}{\rho_4} z+\frac{\gamma_3}{\rho_4} \tilde{y}_x-\frac{\gamma_3}{\rho_4}\tilde{z}-\frac{\gamma_4}{\rho_4} A_*^{\beta}\tilde{z}
 \end{array}\right)
\end{equation}
with domain given by
\begin{eqnarray}\label{DOMINIOBcorr}
		\nonumber\mathfrak{D}(\mathbb{B}):=\{U \in \mathcal{H}: (\tilde{u},\tilde{v},\tilde{y},\tilde{z}) \in [D(A^\frac{1}{2})\times D(A_*^\frac{1}{2})]^2; \kappa_1u+\gamma_1\tilde{u}, \kappa_2y+\gamma_3\tilde{y} \in D(A),\\
		b_1v+\gamma_2A_*^{\alpha-1}\tilde{v}, b_2z+\gamma_4A_*^{\beta-1}\tilde{z} \in D(A_*)\}.
	\end{eqnarray}
	Clearly $\mathfrak{D}(\mathbb{B})$ is dense in $\mathcal{H}$. Moreover,  $\mathbb{B}$  is dissipative. In fact, for each $U=(u,\tilde{u},v, \tilde{v},y, \tilde{y},z, \tilde{z})^T\in \mathfrak{D}(\mathbb{B})$, we have
	\begin{equation}\label{disi}
		{\rm Re}\dual{\mathbb{B}U}{U}_{\mathcal{H}}=-\gamma_1\|(\tilde{u}_x-\tilde{v})\|^2 -\gamma_2\|A_*^{\alpha/2} \tilde{v}\|^2 - \gamma_3\|(\tilde{y}_x-\tilde{z})\|^2 -\gamma_4 \|A_*^{\beta/2}\tilde{z}\|^2 \leq 0.
	\end{equation}
	Therefore, it is enough to show that  $0\in \rho(\mathbb{B})$. Thus, we must show that $(0I-\mathbb{B})^{-1}$ exists and is bounded in $\mathcal{H}$.
	Given $F=(f^1, f^2, f^3,f^4,f^5,f^6,f^7, f^8)^T\in \mathcal{H}$, we look for a unique $U=(u,\tilde{u},v,\tilde{
		v},y,\tilde{y},z,\tilde{z})^T\in \mathfrak{D}(\mathbb{B})$,  such that
	\begin{equation}\label{equa}
		\mathbb{B}(U)=F \qquad {\rm in}\qquad \mathcal{H}.
	\end{equation}
	Equivalently, 
	\begin{eqnarray}
		\tilde{u} = f^1 ,  \tilde{y}=f^5 &\hbox{ in }& D(A^{1/2})\label{pv 1},\\
		\tilde{v}  = f^3, \tilde{z} = f^7 &\hbox{ in } &D(A_*^{1/2})\label{pv 2},
	\end{eqnarray}
	\begin{eqnarray}
		\kappa_1 (u_x-v)_x +m (y - u)+\gamma_1f_{xx}^1=\gamma_1f^3_x+\rho_1f^2 &\hbox{ in }& D(A^{0}),\label{pvf 1}\\
		-b_1 A_*v+\kappa_1 (u_x-v) -\gamma_2A_*^{\alpha}f^3= -\gamma_1(f_x^1-f^3)+\rho_2f^4 &\hbox{ in }& D(A_*^{0}),\label{pvf 2}\\
		\kappa_2 (y_x-z)_x -m (y - u)+\gamma_3f_{xx}^5=-\gamma_3f^7_x+\rho_3f^6 &\hbox{ in }& D(A^{0}), \label{pvf 3}\\
		-b_2 A_*z+\kappa_2 (y_x-z)-\gamma_4A_*^{\beta}f^7 = -\gamma_3(f_x^5-f^7)+\rho_4f^8 &\hbox{ in }& D(A_*^{0}). \label{pvf 4}
	\end{eqnarray}
	Hence, the variational problem of the system \eqref{pvf 1}-\eqref{pvf 4} is given by
	$$B((u,v,y,z),(u^*,v^*,y^*,z^*))=\mathcal{L}(u^*,v^*,y^*,z^*),$$
	where  $B(\cdot,\cdot)$ is a sesquilinear form in  $[D(A^\frac{1}{2})\times D(A_*^\frac{1}{2})]^2$
	defined by
	\begin{eqnarray*}
		B((u,v,y,z),(u^*,v^*,y^*,z^*)) 
		=  \kappa_1 \dual{u_x-v}{u_x^*-v^*} +m\dual{ y - u}{y^*-u^*}\\+\kappa_2 \dual{y_x-z}{y^*_x-z^*} 
		 +b_1 \dual{A_*^{1/2}v}{A_*^{1/2}v^*} + b_2\dual{ A_*^{1/2}z}{A_*^{1/2}z^*}.
	\end{eqnarray*}
	Considering the norm
	\begin{eqnarray*}\label{Norma-sesquilinear}
		\|(u^*,v^*,y^*,z^*)\|^2_{[D(A^\frac{1}{2})\times D(A_*^\frac{1}{2})]^2} 
		=  \kappa_1 \|u_x^*-v^*\|^2 +m\|y^*-u^*\|^2\\+\kappa_2 \|y^*_x-z^*\|^2 +b_1 \|A_*^{1/2}v^*\|^2 + b_2\|A_*^{1/2}z^*\|^2,
	\end{eqnarray*}
it follows that	$B$ is strongly  coercive in $[D(A^\frac{1}{2})\times D(A_*^\frac{1}{2})]^2$ because 
	\begin{eqnarray*}\label{fortemente-coerciva}B((u,v,y,z),(u,v,y,z))=\|(u,v,y,z)\|^2_{[D(A^\frac{1}{2})\times D(A_*^\frac{1}{2})]^2}.
	\end{eqnarray*}
Due to
	\begin{eqnarray*}
	|\mathcal{L}(u^*,v^*,y^*,z^*)| &\leq& C \|F\|_{\mathcal{H}} \|(u^*,v^*,y^*,z^*)\|_{[D(A^\frac{1}{2})\times D(A_*^\frac{1}{2})]^2}\label{continua}
\end{eqnarray*}
	we prove that the linear form 
	$$\mathcal{L}: [D(A^\frac{1}{2})\times D(A_*^\frac{1}{2})]^2 \rightarrow \mathbb{C}$$
	defined by
	\begin{eqnarray*}
		\mathcal{L}(u^*,v^*,y^*,z^*) &=&-\gamma_1\dual{f_x^1-f^3}{u^*_x -v^*}-\gamma_3\dual{f_x^5-f^7}{y^*_x-z^*}\\
		&&-\gamma_2\dual{A_*^{\alpha/2}f^3}{A_*^{\alpha/2}v^*} -\gamma_4\dual{A_*^{\beta/2}f^7}{A_*^{\beta/2}z^*}\\
		&&-\rho_1\dual{f^2}{u^*} -\rho_2\dual{f^4}{v^*} -\rho_3\dual{f^6}{y^*}-\rho_4\dual{f^8}{z^*}
	\end{eqnarray*}
	is continuous. 	By the Lax-Milgran Theorem there is a unique $(u,v,y,z) \in [D(A^\frac{1}{2})\times D(A_*^\frac{1}{2})]^2$ such that
	\begin{eqnarray}\label{lax-milgran}
		B((u,v,y,z),(u^*,v^*,y^*,z^*))=\mathcal{L}(u^*,v^*,y^*,z^*),
	\end{eqnarray}
	for all $(u^*,v^*,y^*,z^*) \in [D(A^\frac{1}{2})\times D(A_*^\frac{1}{2})]^2.$
	
	Choosing 
	$(u^*,0,0,0), (0,v^*,0,0),$ $(0,0,y^*,0),  (0,0,0,z^*) \in [D(A^\frac{1}{2})\times D(A_*^\frac{1}{2})]^2$, it is easy to see, that
	\begin{eqnarray*}
		\tilde{u},  \tilde{y} \in  D(A^{1/2}), \,\,
		\tilde{v}, \tilde{z} \in D(A_*^{1/2}),\,\,
		\kappa_1 u + \gamma_1\tilde{u} \in D(A),\\
		b_1 v  +\gamma_2A_*^{\alpha-1}\tilde{v} \in D(A_*),\,\,
		\kappa_2 y +\gamma_3\tilde{y} \in D(A),\,\,
		b_2 z+ \gamma_4 A_*^{\beta-1}\tilde{z} \in D(A_*),
	\end{eqnarray*}
	consequently 
	$U=(u,\tilde{u},v,\tilde{
		v},y,\tilde{y},z,\tilde{z}) \in \mathfrak{D}(\mathbb{B})$ and
	 $\mathbb{B}$ is subjective. On the other hand, due to
	 $$\|U\|^2_{\mathcal{H}}=\|\mathbb{B}^{-1}F\|^2_{\mathcal{H}}\leq C\|F\|^2_{\mathcal{H}},$$
	 it follows that $\mathbb{B}^{-1}$ is a bounded operator. Therefore, we conclude that $0\in \rho(\mathbb{B})$.

\section{Exponential Decay,  for $(\alpha,\beta)\in [0,1]^2$}\label{3.1}

\paragraph*{}In this section, we will study the asymptotic behavior of the semigroup of the problem \eqref{Eq1.1}-\eqref{Eq1.9}.  We will use the following spectral characterization of exponential stability of semigroups due to Gearhart\cite{Gearhart}.
\begin{theorem}[See \cite{LiuZ}, Theorem 1.3.2]\label{LiuZExponential}
Let $S(t)=e^{t\mathbb{B}}$ be  a  $C_0$-semigroup of contractions on  a Hilbert space $ \mathcal{H}$. Then $S(t)$ is exponentially stable if and only if  
	\begin{equation}\label{EImaginario}
\rho(\mathbb{B})\supseteq\{ i\lambda/ \lambda\in \R \} 	\equiv i\R
\end{equation}
and
\begin{equation}\label{Exponential}
 \limsup\limits_{|\lambda|\to
   \infty}   \|(i\lambda I-\mathbb{B})^{-1}\|_{\mathcal{L}( \mathcal{H})}<\infty
\end{equation}
holds.
\end{theorem}
\begin{remark}
 Note that to show the condition \eqref{Exponential}, it is enough to proof that: given a real number $\delta>0$, there exists a constant $C_\delta>0$ such that the solutions of the problem \eqref{Eq1.1}-\eqref{Eq1.9}  satisfy the following inequality
 \begin{equation}\label{EqvExponencial}
 \|U\|^2_{\mathcal{H}}\leq C_\delta\|F\|_{\mathcal{H}}\|U\|_{\mathcal{H}}\qquad {\rm for}\quad 0\leq\alpha,\beta\leq 1.
 \end{equation}
 \end{remark}
\paragraph*{} In order to obtain \eqref{EqvExponencial}, we need a priori estimates. Consider
 $$U=(u,\tilde{u},v,\tilde{v},y,\tilde{y},z, \tilde{z})^T\in \mathfrak{D}(\mathbb{B})\,\,{\rm and}\,\, F=(f^1, f^2, f^3, f^4, f^5, f^6, f^7, f^8)^T\in \mathcal{H}$$  such that $(i\lambda I-\mathbb{B})U=F$, with $\lambda\in \R$. The problem \eqref{Eq1.1}-\eqref{Eq1.9} can be rewritten in the form:
\begin{eqnarray}
	\label{Eq001AE}
	i\lambda u-\tilde{u} & = & f^1\quad {\rm in}\quad D(A^\frac{1}{2}),\\
	\label{Eq002AE}
	i\lambda\tilde{u}+\dfrac{\kappa_1}{\rho_1}Au+\dfrac{\kappa_1 }{\rho_1}v_x-\dfrac{m}{\rho_1}(y-u)+\dfrac{\gamma_1}{\rho_1}A\tilde{u}+\dfrac{\gamma_1}{\rho_1}\tilde{v}_x & = & f^2\quad {\rm in}\quad D(A^0),\\
	\label{Eq003AE}
	i\lambda v-\tilde{v} & = & f^3\quad {\rm in}\quad D(A_*^\frac{1}{2}),\\
	\label{Eq004AE}
	i\lambda \tilde{v}+\dfrac{b_1}
	{\rho_2}A_*v-\dfrac{\kappa_1}{\rho_2}u_x+\dfrac{\kappa_1}{\rho_2}v-\dfrac{\gamma_1}{\rho_2}\tilde{u}_x+\dfrac{\gamma_1}{\rho_2}\tilde{v}+\dfrac{\gamma_2}{\rho_2}A^\alpha_*\tilde{v} & = & f^4\quad {\rm in}\quad D(A^0_*),\\
	\label{Eq005AE}
	i\lambda y-\tilde{y} &= & f^5\quad {\rm in}\quad D(A^\frac{1}{2}),\\
	\label{Eq006AE}
	i\lambda\tilde{y}+\dfrac{\kappa_2}{\rho_3}Ay+\dfrac{\kappa_2}{\rho_3}z_x+\dfrac{m}{\rho_3}(y-u)+\dfrac{\gamma_3}{\rho_3}A\tilde{y}+\dfrac{\gamma_3}{\rho_3}\tilde{z}_x & = & f^6\quad {\rm in}\quad D(A^0),\\
	\label{Eq007AE}
	i\lambda z-\tilde{z} &= & f^7\quad {\rm in}\quad D(A_*^\frac{1}{2}),\\
	\label{Eq008AE}
	i\lambda\tilde{z}+\dfrac{b_2}{\rho_4}A_*z-\dfrac{\kappa_2}{\rho_4}y_x+\dfrac{\kappa_2}{\rho_4}z-\dfrac{\gamma_3}{\rho_4}\tilde{y}_x+\dfrac{\gamma_3}{\rho_4}\tilde{z}+\dfrac{\gamma_4}{\rho_4}A^\beta_*\tilde{z} &= & f^8\quad {\rm in}\quad D(A_*^0).
\end{eqnarray}
Due to  \eqref{disi}, we estimate
\begin{equation}\label{dis-10}
\gamma_1 \|\tilde{u}_x-\tilde{v}\|^2+\gamma_2\|A^\frac{\alpha}{2}_*\tilde{v}\|^2+\gamma_3\| \tilde{y_x}-\tilde{z}\|^2+\gamma_4\|A^\frac{\beta}{2}_* \tilde{z}\|^2 \leq  \|F\|_{\mathcal{H}}\|\|U\|_{\mathcal{H}}.
\end{equation}
Since $\alpha\geq0$ e $\beta\geq 0$, it follows that
\begin{equation}\label{Estima01DExp}
	\rho_2\|\tilde{v}\|^2+\rho_4\|\tilde{z}\|^2\leq C\|F\|_{\mathcal{H}}\|U\|_{\mathcal{H}}.
	\end{equation}
Next, we proof the following lemmas:
\begin{lemma}\label{Lemma1} Suppose $\alpha,\beta$ belongs to $[0,1]$. Given a real number $\delta>0$, there exists a constant $C_\delta>0$ such that the solutions of the problem \eqref{Eq1.1}-\eqref{Eq1.9}  satisfy the following inequality
	\begin{equation}
		\label{Estima02Dexp}
		|\lambda|\| y-u\|^2  \leq  C_\delta\|F\|_{\mathcal{H}}\|U\|_{\mathcal{H}},
	\end{equation}
for all $|\lambda|>\delta$.
\end{lemma}
\begin{proof}
	 Taking into account \eqref{Eq001AE} and \eqref{Eq005AE}, we have 
		\begin{equation}\label{key}
		i\lambda(y-u)-(\tilde{y}-\tilde{u})=f^5-f^1.
		\end{equation}
		Taking the duality product between \eqref{key} and $y-u$, we obtain
		\begin{equation}\label{Eq01Lemma3}
			i\lambda\|y-u\|^2=\dual{\tilde{y}}{y-u}-\dual{\tilde{u}}{y-u}+\dual{f^5}{y-u}-\dual{f^1}{y-u}.
		\end{equation}
		By the Cauchy-Schwarz and Young inequalities with $\varepsilon>0$, we estimate
		\begin{equation*} 
			|\lambda| \|y-u\|^2\leq C_\varepsilon\{ \|\tilde{y}\|^2+\|\tilde{u}\|^2\}+\varepsilon\|y-u\|^2+C_\delta\|F\|_{\mathcal{H}}\|U\|_{\mathcal{H}}.
		\end{equation*}
		Due to \eqref{dis-10}, \eqref{Estima01DExp}, we get
	\begin{equation}\label{Estima03Dexp}
			\|\tilde{u}_x\|^2\leq C\|F\|_{\mathcal{H}}\|U\|_{\mathcal{H}}+\|\tilde{v}\|^2\leq C\|F\|_{\mathcal{H}}\|U\|_{\mathcal{H}}
			\end{equation}
		and	
			\begin{equation}\label{Estima04Dexp}
				\|\tilde{y}_x\|^2\leq C\|F\|_{\mathcal{H}}\|U\|_{\mathcal{H}}+\|\tilde{z}\|^2 \leq C\|F\|_{\mathcal{H}}\|U\|_{\mathcal{H}}.
				\end{equation}
			Making use of \eqref{dis-10} and  considering $|\lambda|>\delta>1$, we conclude the proof of Lemma \ref{Lemma1}.
		\end{proof}
Note that, from estimates \eqref{Estima02Dexp}-\eqref{Estima04Dexp}, we have
\begin{equation}\label{Estima05Dexp}
\rho_1\|\tilde{u}\|^2+\rho_3\|\tilde{y}\|^2+m\|y-u\|^2\leq C\|F\|_{\mathcal{H}}\|U\|_{\mathcal{H}},
\end{equation}
for all $|\lambda|\geq 1$.
Therefore, to finish proving the estimate \eqref{EqvExponencial}, it remains to obtain the estimate:
$$\kappa_1\|u_x-v\|^2+ b_1\|A^\frac{1}{2}_* v\|^2+\kappa_2\|y_x-z\|^2+b_2\|A^\frac{1}{2}_*z\|^2\leq C\|F\|_{\mathcal{H}}\|U\|_{\mathcal{H}}.$$

\begin{lemma}\label{Lemma2}
Suppose $\alpha,\beta$ belongs to $[0,1]$. Given a real number $\delta>0$, there exists a constant $C_\delta>0$ such that the solutions of the problem \eqref{Eq1.1}-\eqref{Eq1.9}  satisfy the following inequality
\begin{eqnarray*}
\label{Item02Lemma2}
(i)\quad \kappa_1\|u_x-v\|^2+b_1\|A^\frac{1}{2}_*v\|^2\leq \varepsilon\|U\|^2_{\mathcal{H}}+C_\delta\|F\|_{\mathcal{H}}\|U\|_{\mathcal{H}},\\
\label{Item03Lemma2}
(ii)\quad \kappa_2\|y_x-z\|^2+b_2\|A^\frac{1}{2}_*z\|^2\leq \varepsilon\|U\|^2_{\mathcal{H}}+C_\delta\|F\|_{\mathcal{H}}\|U\|_{\mathcal{H}},
\end{eqnarray*}
for all $|\lambda|>\delta$ and $\varepsilon>0$.
\end{lemma}
\begin{proof}
	$\! (i)$  Taking the duality product between \eqref{Eq002AE} and $\rho_1u$ and using \eqref{Eq001AE}, we have
	\begin{equation}\label{a}
	\kappa_1\Big[ \|A^\frac{1}{2}u\|^2-\dual{v}{u_x}\Big]=\rho_1\|\tilde{u}\|^2+\rho_1\dual{\tilde{u}}{f^1}+m\dual{(y-u)}{u}-\gamma_1\dual{A\tilde{u}}{u}-\gamma_1\dual{\tilde{v}_x}{u}+\rho_1\dual{f^2}{u}.
	\end{equation}
Similarly, taking the duality product between \eqref{Eq004AE} and $\rho_2v$ and using \eqref{Eq003AE}, we obtain
\begin{eqnarray}
b_1\|A^\frac{1}{2}_*v\|^2+\kappa_1\Big[\|v\|^2-\dual{u_x}{v}\Big]= \rho_2\|\tilde{v}\|^2+\rho_2\dual{\tilde{v}}{f^3}\notag\\ +\gamma_1\dual{\tilde{u}_x}{v}-\gamma_1\dual{\tilde{v}}{v}-\gamma_2\dual{A^\alpha_*\tilde{v}}{v}+\rho_2\dual{f^4}{v}.\label{b}
\end{eqnarray}
Summing \eqref{a} and \eqref{b}, we find
\begin{eqnarray}
\kappa_1\|u_x-v\|^2+b_1\|A^\frac{1}{2}_*v\|^2=\rho_2\|\tilde{v}\|^2+\rho_1\|\tilde{u}\|^2+\rho_1\dual{\tilde{u}}{f^1}+\underbrace{m\dual{(y-u)}{u}}_{I_1}-i\lambda\gamma_1\|A^\frac{1}{2}u\|^2\notag\\
+\gamma_1\dual{A^\frac{1}{2}f^1}{A^\frac{1}{2}u}\underbrace{-\gamma_1\dual{\tilde{v}_x}{u}}_{I_2}
+\rho_1\dual{f^2}{u}+\rho_2\dual{\tilde{v}}{f^3}
+\underbrace{\gamma_1\dual{\tilde{u}_x}{v}}_{I_3}-i\lambda\gamma_1\|v\|^2\notag\\
+\gamma_1\dual{f^3}{v}
-i\lambda\gamma_2\|A^\frac{\alpha}{2}_*v\|^2+\gamma_2\dual{A^\frac{\alpha}{2}_*f^3}{A^\frac{\alpha}{2}_*v}+  \rho_2\dual{f^4}{v}.\label{Eq005Dexp}
\end{eqnarray}
We calculate
	\begin{equation}\label{Eq006Dexp}
		I_1=m\dual{(y-u)}{u}=-m\dual{y-u}{y-u-y}=-m\|y-u\|^2+m\dual{y-u}{y},
	\end{equation}
	\begin{eqnarray}
	I_2=-\gamma_1\dual{\tilde{v}_x}{u}=\gamma_1\dual{\tilde{v}}{u_x-v+v}=\gamma_1\dual{\tilde{v}}{u_x-v}+\gamma_1\dual{\tilde{v}}{v}\notag\\
		=\gamma_1\dual{\tilde{v}}{u_x-v}+\gamma_1\dual{i\lambda v-f^3}{v}=\gamma_1\dual{\tilde{v}}{u_x-v}+i\lambda\gamma_1\|v\|^2-\gamma_1\dual{f^3}{v}\label{Eq007Dexp}
	\end{eqnarray}
and	
	\begin{equation}\label{Eq008Dexp}
		I_3=\gamma_1\dual{\tilde{u}_x}{v}=-\gamma_1\dual{\tilde{u}}{v_x}.
	\end{equation}
Substituting \eqref{Eq006Dexp}-\eqref{Eq008Dexp} into \eqref{Eq005Dexp} and taking the real part, we get
\begin{eqnarray*}
\kappa_1\|u_x-v\|^2+b_1\|A^\frac{1}{2}_*v\|^2=\rho_2\|\tilde{v}\|^2+\rho_1\|\tilde{u}\|^2-m\|y-u\|^2\\+\text{Re}\Big\{ \rho_1\dual{\tilde{u}}{f^1}+m\dual{y-u}{y}
+\gamma_1\dual{A^\frac{1}{2}f^1}{A^\frac{1}{2}u}+ \gamma_1\dual{\tilde{v}}{u_x-v}
+\rho_1\dual{f^2}{u}\\+\rho_2\dual{\tilde{v}}{f^3}
-\gamma_1\dual{\tilde{u}}{v_x}
+\gamma_2\dual{A^\frac{\alpha}{2}_*f^3}{A^\frac{\alpha}{2}_*v}+  \rho_2\dual{f^4}{v}\Big\}.
\end{eqnarray*}
By the Cauchy-Schwarz and Young inequalities with $\varepsilon>0$, we estimate
 \begin{eqnarray*}
 	\kappa_1\|u_x-v\|^2+b_1\|A^\frac{1}{2}_*v\|^2  \leq  C\Big\{ \|\tilde{v}\|^2+\|\tilde{u}\|^2+\|y-u\|^2+\|\tilde{u}\|\|f^1\| +\|A^\frac{1}{2} f^1\|\|A^\frac{1}{2}u\|\\
 + \|f^2\|\|u\|+\|\tilde{v}\|\|f^3\|
 	+\|f^4\|\|v\|+\|A^\frac{\alpha}{2}_*f^3\|\|A^\frac{\alpha}{2}_*v\|
 	\Big\} \\
+ C_\varepsilon\{\|\tilde{u}\|^2+\|\tilde{v}\|^2+\|y-u\|^2\}+  \varepsilon\{ \|v_x\|^2+\|u_x-v\|^2+\|y\|^2 \}.
 \end{eqnarray*}
From Lemma \ref{Lemma1} and $|\lambda|\geq 1$, we have
\begin{equation*}
 \|y\|^2-\|u\|^2\leq \|y-u\|^2\leq |\lambda|\|y-u\|\leq  C_\delta \|F\|_\mathcal{H}\|U\|_\mathcal{H}
 \end{equation*}
Consequently
 \begin{equation}\label{Eq009Dexp}
 	\varepsilon\|y\|^2\leq  C\|F\|_\mathcal{H}\|U\|_\mathcal{H}+\varepsilon\|U\|^2_\mathcal{H}.
 \end{equation}
As $\|A^\frac{1}{2}_*v \|^2=\|v_x\|^2$.   The proof of item $(i)$ is complete. \\
 {\bf Proof:} $(ii)$ Taking the duality product between \eqref{Eq006AE} and $\rho_3y$ and using \eqref{Eq005AE}, we have
\begin{eqnarray}
\kappa_2\dual{(y_x-z)}{y_x}=\rho_3\dual{\tilde{y}}{i\lambda y}-m\dual{(y-u)}{y}-\gamma_3\dual{A\tilde{y}}{y}-\gamma_3\dual{\tilde{z}_x}{y}+\rho_3\dual{f^6}{y}=\notag\\
\rho_3\|\tilde{y}\|^2+\rho_3\dual{\tilde{y}}{f^5}-m\dual{y-u}{y}-i\lambda\gamma_3\|A^\frac{1}{2}y\|^2+\gamma_3\dual{A^\frac{1}{2}f^5}{A^\frac{1}{2}y}
-\gamma_3\dual{\tilde{z}_x}{y}+\rho_3\dual{f^6}{y}.\label{c}
\end{eqnarray}
Similarly, taking the duality product of \eqref{Eq008AE} for $\rho_4z$ and using \eqref{Eq007AE}, we obtain
\begin{eqnarray}
b_2\|A_*^\frac{1}{2}z\|^2-\kappa_2\Big[ \dual{(y_x-z)}{z} \Big]=\rho_4\|\tilde{z}\|^2+\rho_4\dual{\tilde{z}}{f^7}+\gamma_3\dual{\tilde{y}_x}{z}\notag\\-i\dfrac{\gamma_3}{\lambda}\|\tilde{z}\|^2
-i\dfrac{\gamma_3}{\lambda}\dual{\tilde{z}}{f^7}-i\lambda\gamma_4\|A^\frac{\beta}{2}_*z\|^2+\gamma_4\dual{A^\frac{\beta}{2}_*f^7}{A^\frac{\beta}{2}_*z}+\rho_4\dual{f^8}{z}.\label{d}
\end{eqnarray}
Summing \eqref{a} and \eqref{b}, we find
\begin{eqnarray}
\kappa_2\|y_x-z\|^2+b_2\|A^\frac{1}{2}_*z\|^2=\rho_3\|\tilde{y}\|^2+\rho_3\dual{\tilde{y}}{f^5}-m\dual{y-u}{y}-i\lambda\gamma_3\|A^\frac{1}{2}y\|^2\notag\\+\gamma_3\dual{A^\frac{1}{2}f^5}{A^\frac{1}{2}y}
-\gamma_3\dual{\tilde{z}_x}{y}+\rho_3\dual{f^6}{y}+\rho_4\|\tilde{z}\|^2+\rho_4\dual{\tilde{z}}{f^7}+\gamma_3\dual{\tilde{y}_x}{z}-i\dfrac{\gamma_3}{\lambda}\|\tilde{z}\|^2 \notag \\
-i\dfrac{\gamma_3}{\lambda}\dual{\tilde{z}}{f^7}-i\lambda\gamma_4\|A^\frac{\beta}{2}_*z\|^2+\gamma_4\dual{A^\frac{\beta}{2}_*f^7}{A^\frac{\beta}{2}_*z}+\rho_4\dual{f^8}{z}.\label{Eq010Dexp}
\end{eqnarray}
Note that 
\begin{eqnarray}
-\gamma_3\dual{\tilde{z}_x}{y}=\gamma_3\dual{\tilde{z}}{y_x-z+z}=\gamma_3\dual{\tilde{z}}{y_x-z}+\gamma_3\dual{i\lambda z-f^7}{z}\notag \\
=\gamma_3\dual{\tilde{z}}{y_x-z}+i\lambda\gamma_3\|z\|^2-\gamma_3\dual{f^7}{z}.\label{Eq011Dexp}
\end{eqnarray}
Substituting \eqref{Eq011Dexp} into \eqref{Eq010Dexp},  taking the real part and   applying Cauchy-Schwarz inequality, we obtain for $|\lambda|\geq 1$ that
\begin{eqnarray}
		\kappa_2\|y_x-z\|^2+b_2\|A^\frac{1}{2}_*z\|^2\leq \rho_3\|\tilde{y}\|^2+\rho_4\|\tilde{z}\|^2+\underbrace{ m|\dual{y-u}{y}|}_{I_1}+\underbrace{\gamma_3|\dual{\tilde{y}_x}{z}|}_{I_2}
+\underbrace{\gamma_3|\dual{\tilde{z}}{y_x-z}|}_{I_3}\notag\\+C\{\|A^\frac{1}{2}f^5\|\|A^\frac{1}{2}y\| +\|f^6\|\|y\| +\| \tilde{y}\|\|f^5\|+\| \tilde{z}\|\|f^7\|
+\|f^7\|\|z\|+\|f^7\|\|\tilde{z}\|\notag\\
+\|A^\frac{\beta}{2}_*f^7\|\|A^\frac{\beta}{2}_*z\|+\|f^8\| \|z\|\}.\label{Eq03Lemma2A}
\end{eqnarray}
By the Young inequality with $\varepsilon>0$, we estimate
 \begin{equation}\label{Eq012Dexp}
 I_1=m|\dual{y-u}{y}|\leq C_\varepsilon\|y-u\|^2+\varepsilon\|y\|^2 \leq  C\|U\|_\mathcal{H} \|F\|_\mathcal{H}+\varepsilon\|U\|^2_\mathcal{H},
 \end{equation}
 \begin{equation}\label{Eq013Dexp}
 I_3=\gamma_3|\dual{\tilde{z}}{y_x-z}|\leq C_\varepsilon\|\tilde{z}\|^2+\varepsilon\|y_x-z\|^2\leq  C\|U\|_\mathcal{H} \|F\|_\mathcal{H}+\varepsilon\|U\|^2_\mathcal{H}.
 \end{equation}
By the Young inequality with $\varepsilon>0$ and taking into account \eqref{Estima04Dexp}, we estimate
\begin{equation}\label{Eq014Dexp}
I_2= \gamma_3|\dual{\tilde{y}_x}{z}|\leq C_\varepsilon\|\tilde{y}_x\|^2+\varepsilon \|z\|^2 \leq C\|F\|_{\mathcal{H}}\|U\|_{\mathcal{H}}+\varepsilon\|U\|^2_{\mathcal{H}}.
 \end{equation}
Substituting \eqref{Eq012Dexp}, \eqref{Eq013Dexp}, \eqref{Eq014Dexp} into \eqref{Eq03Lemma2A} and combining \eqref{dis-10}, \eqref{Estima04Dexp} and \eqref{Eq011Dexp}, we complete the proof of item $(ii)$.
\end{proof}
\begin{lemma}\label{EImaginary}
	Let $\rho(\mathbb{B})$ be the resolvent set of operator
	$\mathbb{B}$. Then
	\begin{equation}
		i\hspace{0.5pt}\mathbb{R}\subset\rho(\mathbb{B}).
	\end{equation}
\end{lemma}
\begin{proof}
	Since $\mathbb{B}$  is the infinitesimal generator of a $C_0-$semigroup of contractions $S(t)$, $ t\geq 0$,  from Theorem \ref{THY},  $\mathbb{B}$ is a closed operator and $\mathfrak{D}(\mathbb{B})$ has compact embedding into the energy space $\mathcal{H}$, the spectrum $\sigma(\mathbb{B})$ contains only eigenvalues.
	Let us prove that $i\R\subset\rho(\mathbb{B})$ by using an argument by contradiction, so we suppose that $i\R\not\subset \rho(\mathbb{B})$. 
	As $0\in\rho(\mathbb{B})$ and $\rho(\mathbb{B})$ is open, we consider the highest positive number $\lambda_0$ such that the $(-i\lambda_0,i\lambda_0)\subset\rho(\mathbb{B})$ then $i\lambda_0$ or $-i\lambda_0$ is an element of the spectrum $\sigma(\mathbb{B})$.   
	Suppose $i\lambda_0\in \sigma(\mathbb{B})$ (if $-i\lambda_0\in \sigma(\mathbb{B})$ the proceeding is similar), then for $0<\delta<\lambda_0$ there exist a sequence of real numbers $(\lambda_n)$, with $\delta\leq\lambda_n<\lambda_0$, $\lambda_n\con \lambda_0$, and a vector sequence  $U_n=(u_n,\tilde{u}_n,v_n,\tilde{v}_n, y_n,\tilde{y}_n, z_n, \tilde{z}_n)\in \mathfrak{D}(\mathbb{B})$ with  unitary norms, such that
	\begin{equation*}
		\|(i\lambda_nI-\mathbb{B}) U_n\|_{\mathcal{H}}=\|F_n\|_{\mathcal{H}}\con 0,
	\end{equation*}
	as $n\con \infty$.   Due to \eqref{EqvExponencial}, we have
	\begin{eqnarray*}
		\|U_n\|^2_{\mathcal{H}}=\bigg[ \rho_1 |\tilde{u}_n\|^2+\rho_2\|\tilde{v}_n\|^2+\rho_3\|\tilde{y}_n\|^2+  \rho_4\|\tilde{z}_n\|^2+ \kappa_1\|u_{nx}-v_n\|^2+\kappa_2\|y_{nx}-z_n\|^2\\+m\|y_n-u_n\|^2 
		+b_1\|A_*^{1/2}v_n\|^2 +b_2 \|A_*^{1/2}z_n\|^2 \bigg ]
		\leq C_\delta\|F_n\|_{\mathcal{H}}\|U_n\|_{\mathcal{H}}=C_\delta\|F_n\|_{\mathcal{H}}\con 0.
	\end{eqnarray*}
	Therefore, we obtain  $\|U_n\|_{\mathcal{H}}\con 0$ but this is absurd, since $\|U_n\|_{\mathcal{H}}=1$ for all $n\in\N$. Thus, $i\R\subset \rho(\mathbb{B})$. This completes the proof of this lemma. 
\end{proof}
\begin{theorem}\label{TDExponential}
The semigroup $S(t) = e^{t\mathbb{B}}$ is exponentially stable as long as the parameters $\alpha$ and $\beta$ belongs to $[0,1]$.
\end{theorem}
\begin{proof} 
According to  Lemmas \ref{Lemma1} and  \ref{Lemma2}  and applying in the sequence the estimates of \eqref{dis-10}, we get
\begin{equation*}\label{Eq012Exponential}
\|U\|_\mathcal{H}^2\leq C_\delta \|F\|_{\mathcal{H}}\|U\|_{\mathcal{H}}\quad{\rm for}\quad 0\leq\alpha,\beta\leq 1.
\end{equation*}
This implies the condition \eqref{Exponential}. Moreover, by Lemma \ref{EImaginary} the condition \eqref{EImaginario} is satisfied. 
\end{proof}
\section{Analyticity,  for $(\alpha,\beta)\in [0,1]^2$ }
In this section, we will show that the semigroup $S(t)=e^{t\mathbb{B}}$ is analytical for parameters $(\alpha,\beta)\in [0,1]^2$. The proof will be performed using the characterization given in the following Theorem: 

\begin{theorem}[See \cite{LiuZ}]\label{LiuZAnalyticity}
    Let $S(t)=e^{t\mathbb{B}}$ be a $C_0$-semigroup of contractions  on a Hilbert space $\mathcal{H}$.   Suppose that
    \begin{equation}\label{RBEixoIm}
    \rho(\mathbb{B})\supseteq\{ i\lambda/ \lambda\in \R \}  \equiv i\R.
    \end{equation}
     Then $S(t)$ is analytic if and only if
    \begin{equation}\label{Analyticity}
     \limsup\limits_{|\lambda|\to
        \infty}
    \|\lambda(i\lambda I-\mathbb{B})^{-1}\|_{\mathcal{L}(\mathcal{H})}<\infty 
    \end{equation}
    holds.
\end{theorem}
We will first prove some lemmas:
\begin{lemma}\label{Lemma9}
Suppose $\alpha,\beta$ belongs to $[0,1]$. Given a real number $\delta>0$, there exists a constant $C_\delta>0$ such that the solutions of the problem \eqref{Eq1.1}-\eqref{Eq1.9}  satisfy the following inequalities
\begin{eqnarray*}\label{Item01Lemma9}
(i)\quad\quad\; |\lambda|\|A_*^\frac{1}{2}v\|^2 + |\lambda|\|\tilde{v}\|^2& \leq &  C_\delta\|F\|_\mathcal{H}\|U\|_\mathcal{H},\\
\label{Item02Lemma9}
(ii)\quad\quad\; |\lambda|\|A_*^\frac{1}{2}z\|^2 + |\lambda|\|\tilde{z}\|^2& \leq &  C_\delta\|F\|_\mathcal{H}\|U\|_\mathcal{H},\\
\label{Item03Lemma9}
(iii)\quad \quad\quad\quad\quad\quad\,\,\,\,\quad |\lambda|\|\tilde{u}\|^2 &\leq & C_\delta \|F\|_\mathcal{H}\|U\|_\mathcal{H},\\
\label{Item04Lemma9}
(iv)\quad\quad\quad\quad\quad\quad\quad\,\,\ |\lambda|\|\tilde{y}\|^2 &\leq & C_\delta\|F\|_\mathcal{H}\|U\|_\mathcal{H},
\end{eqnarray*}
for all $|\lambda| > \delta$.
\end{lemma}
\begin{proof} $(i)-(ii)$ Making use of the duality product between equation \eqref{Eq003AE} and $A_*v$, we have
\begin{eqnarray}\label{eq*lema09}
 i\lambda\|A_*^\frac{1}{2}v\|^2 =\dual{\tilde{v}}{A_*v}+\dual{f^3}{A_*v}=\dual{\tilde{v}}{A_*v}+\dual{A_*^\frac{1}{2}f^3}{A_*^\frac{1}{2}v}.
\end{eqnarray}
Due to \eqref{Eq004AE}, it follows that 
\begin{eqnarray}
A_*v =\dfrac{\rho_2}{b_1}f^4 -i\lambda \dfrac{\rho_2}{b_1}\tilde{v}+\dfrac{\kappa_1}{b_1}(u_x-v)+\dfrac{\gamma_1}{b_1}(\tilde{u}_x-\tilde{v})-\dfrac{\gamma_2}{b_1}A^\alpha_*\tilde{v}. \label{eqIIlema09}
\end{eqnarray}
Substituting \eqref{eqIIlema09} into \eqref{eq*lema09}, we obtain
\begin{eqnarray*}
 i\lambda\|A_*^\frac{1}{2}v\|^2 = \dfrac{\rho_2}{b_1}\dual{\tilde{v}}{f^4} -i\lambda \dfrac{\rho_2}{b_1}\|\tilde{v}\|^2+\dfrac{\kappa_1}{b_1}\dual{\tilde{v}}{u_x-v}+\dfrac{\gamma_1}{b_1}\dual{\tilde{v}}{\tilde{u}_x-\tilde{v}}-\dfrac{\gamma_2}{b_1}\|A^\frac{\alpha}{2}_*\tilde{v}\|^2 +\dual{A_*^\frac{1}{2}f^3}{A_*^\frac{1}{2}v},
\end{eqnarray*}
whence
\begin{eqnarray*}
 \lambda\big[\|A_*^\frac{1}{2}v\|^2+ \dfrac{\rho_2}{b_1}\|\tilde{v}\|^2\big] =Im\Bigg\{ \dfrac{\rho_2}{b_1}\dual{\tilde{v}}{f^4} +\dfrac{\kappa_1}{b_1}\dual{\tilde{v}}{u_x-v}+\dfrac{\gamma_1}{b_1}\dual{\tilde{v}}{\tilde{u}_x-\tilde{v}} +\dual{A_*^\frac{1}{2}f^3}{A_*^\frac{1}{2}v}\Bigg\}.
\end{eqnarray*}
By the Cauchy-Schwarz and Young inequalities with $\varepsilon>0$, we reduce it to the inequality
\begin{eqnarray*}
 |\lambda|\|A_*^\frac{1}{2}v\|^2+ |\lambda|\|\tilde{v}\|^2 & \leq &
  C\Bigg\{\|\tilde{v}\|\|f^4\| +\|A_*^\frac{1}{2}f^3\|\|A_*^\frac{1}{2}v\|\Bigg\}\\
 & & +\varepsilon\|\tilde{v}\|^2+C_\varepsilon\|{u_x-v}\|^2+\varepsilon\|\tilde{v}\|^2+C_\varepsilon\|\tilde{u}_x-\tilde{v}\|^2.
\end{eqnarray*}
Taking into account \eqref{dis-10} and the Theorem \ref{TDExponential}, we conclude
$$|\lambda|\|A_*^\frac{1}{2}v\|^2 + |\lambda|\|\tilde{v}\|^2 \leq  C_\delta\|F\|_\mathcal{H}\|U\|_\mathcal{H}.$$
This completes the proof of item $(i)$. Acting as by the proof of item $(i)$ with \eqref{Eq007AE} and  \eqref{Eq008AE} instead of \eqref{Eq003AE} and \eqref{Eq004AE}, we have the prove of item $(ii)$.
\\
{\bf Proof:} $(iii)-(iv)$ Taking the duality product between \eqref{Eq002AE} and         $\tilde{u}$, we have
  \begin{equation}\label{Eq01ItemiiiL5}
 i\lambda\|\tilde{u}\|^2=-\dfrac{\kappa_1}{\rho_1}\dual{Au}{\tilde{u}}-\dfrac{\kappa_1}{\rho_1}\dual{v_x}{\tilde{u}}+\dfrac{m}{\rho_1}\dual{y-u}{\tilde{u}}-\dfrac{\gamma_1}{\rho_1}\|A^\frac{1}{2}\tilde{u}\|^2-\dfrac{\gamma_1}{\rho_1}\dual{\tilde{v}_x}{\tilde{u}}+\dual{f^2}{\tilde{u}}.
 \end{equation}
 Due to \eqref{Eq001AE}, \eqref{Eq003AE}, we find
 \begin{equation}\label{Eq02ItemiiiL5}
 \dual{Au}{\tilde{u}}=\dual{Au}{i\lambda u-f^1}=-i\lambda\|A^\frac{1}{2}u\|^2-\dual{A^\frac{1}{2}u}{A^\frac{1}{2} f^1}
\end{equation}
and
\begin{equation}\label{k}
  \dual{\tilde{v}_x}{\tilde{u}}=\dual{i\lambda v_x-f^3_x}{\tilde{u}}=i\dual{\sqrt{|\lambda|}v_x}{\dfrac{\lambda}{\sqrt{|\lambda|}}\tilde{u}}-\dual{f^3_x}{\tilde{u}}.
  \end{equation}
Substituting \eqref{Eq02ItemiiiL5}, \eqref{k} into \eqref{Eq01ItemiiiL5} and taking imaginary part, we get
 \begin{multline*}
 \lambda\|\tilde{u}\|^2=\dfrac{\kappa_1}{\rho_1}\lambda\|A^\frac{1}{2}u\|^2+\text{Im}\bigg\{ \dfrac{\kappa_1}{\rho_1}\dual{A^\frac{1}{2} u}{A^\frac{1}{2}f^1}-\dfrac{\kappa_1}{\rho_1}\dual{v_x}{\tilde{u}}+\dfrac{m}{\rho_1}\dual{y-u}{\tilde{u}}\\
 -i\dfrac{\gamma_1}{\rho_1}\dual{\sqrt{|\lambda|}v_x}{\dfrac{\lambda}{\sqrt{|\lambda|}}\tilde{u}}+\dfrac{\gamma_1}{\rho_1}\dual{f^3_x}{\tilde{u}}+\dual{f^2}{\tilde{u}}\bigg\}.
 \end{multline*}
By the Cauchy-Schwarz and Young inequalities with $\varepsilon>0$, we estimate
  \begin{multline*}
 |\lambda|\|\tilde{u}\|^2\leq \dfrac{\kappa_1}{\rho_1}|\lambda|\|A^\frac{1}{2}u\|^2+C\bigg\{ \|A^\frac{1}{2} u\| \|A^\frac{1}{2}f^1\|+\|v_x\|^2+\|\tilde{u}\|^2+\|y-u\|^2
+\|f^3_x\|\|\tilde{u}\|+\|f^2\|\|\tilde{u}\| \bigg\}\\+C_\varepsilon |\lambda|\|v_x\|^2+\varepsilon |\lambda|\|\tilde{u}\|^2.
 \end{multline*} 
Taking into accout the inequalities
$$\|A^\frac{1}{2}u\|=\|u_x\|,\quad \|u_x\|\leq \|u_x-v\|+\|v\|\leq C\|U\|_\mathcal{H},\quad \|v_x\|=\|A^\frac{1}{2}_* v\|$$
and the item $(i)$ this Lemma,  it follows that
$$|\lambda|\|\tilde{u}\|^2 \leq  C_\delta \|F\|_\mathcal{H}\|U\|_\mathcal{H}$$
and the item $(iii)$ is thereby proved. The prove of item $(iv)$ can be proved as in item $(iii)$ with \eqref{Eq005AE}, \eqref{Eq006AE}, \eqref{Eq007AE} instead
of \eqref{Eq001AE}, \eqref{Eq002AE}, \eqref{Eq003AE}. Consequently, the Lemma \ref{Lemma9} is proved.
\end{proof}

\begin{lemma}\label{Lemma10}
Suppose $\alpha,\beta$ belongs to $[0,1]$. Given a real number $\delta>0$, there exists a constant $C_\delta>0$ such that the solutions of the problem \eqref{Eq1.1}-\eqref{Eq1.9}  satisfy the following inequalities	
\begin{eqnarray*}
\label{Item01Lemma10}
(i)\quad |\lambda|\|u_x-v\|^2  & \leq & C_\delta \|F\|_\mathcal{H}\|U\|_\mathcal{H},\\
\label{Item02Lemma5}
(ii)\quad |\lambda|\|y_x-z\|^2 & \leq & C_\delta \|F\|_\mathcal{H}\|U\|_\mathcal{H},
\end{eqnarray*}
for all $|\lambda|>\delta$.
\end{lemma}
\begin{proof}  Due to \eqref{Eq001AE},\eqref{Eq003AE}, we have
\begin{equation}\label{Eq03Lemma05}
i\lambda(u_x-v)-(\tilde{u}_x-\tilde{v})=f^1_x-f^3.
\end{equation} 
The duality product between \eqref{Eq03Lemma05}  and $u_x-v$ gives that
\begin{equation}
i\lambda\|u_x-v\|^2 = \dual{(\tilde{u}_x-\tilde{v})}{u_x-v}+\dual{f_x^1}{(u_x-v)}-\dual{f^3}{(u_x-v)}.
\end{equation}
By the Cauchy-Schwarz and Young inequalities with $\varepsilon>0$, we get
\begin{equation*} 
|\lambda|\|u_x-v\|^2\leq C_\varepsilon\|\tilde{u}_x-\tilde{v}\|^2+\varepsilon\|u_x-v\|^2+\|f_x^1\|\|u_x-v\|+\|f^3\|\|u_x-v\|.
\end{equation*}
Making use of \eqref{dis-10} for $|\lambda|>\delta$, we conclude the proof of item $(i)$. Using \eqref{Eq005AE}, \eqref{Eq007AE} instead of \eqref{Eq001AE}, \eqref{Eq003AE},  the prove of item $(ii)$ can be handled in the same maner as in the proof of item $(i)$.
\end{proof}
\begin{theorem}\label{AnaliticidadeS2}
The semigroup $S(t)=e^{t\mathbb{B}}$  is analytic for $(\alpha,\beta)\in [0, 1]^2$.
\end{theorem}
 \begin{proof}
According to Lemma \ref{EImaginary}, $i\mathbb{R}\subset \rho(\mathbb{B})$ is verified.  Furthermore, combining Lemmas \ref{Lemma1}, \ref{Lemma9} and \ref{Lemma10}, given a real number $\delta>0$, there exists a constant $C_\delta>0$ such that the solutions of the system \eqref{Eq1.1}-\eqref{Eq1.9}  satisfy 
 \begin{equation*}\label{EquiAnalyticity2} 
 |\lambda|\|U\|^2_{\mathcal{H}}\leq C_\delta\|F\|_{\mathcal{H}}\|U\|_{\mathcal{H}}\qquad\text{for}\qquad (\alpha,\beta)\in [0, 1]^2
 \end{equation*}
and for all $|\lambda|>\delta$.
The proof of Theorem \ref{AnaliticidadeS2} is thereby complete.
\end{proof}
\section{New Analiticity Results for Timoshenko Systems}
The results presented in this work imply new analyticity results for Timoshenko systems: If we disregard the coupling (Van der Waals force $m(y-u)$) our system becomes a decoupled system of two identical Timoshenko systems. Let us consider only the first one:
\begin{eqnarray}
	\label{TEq1.1}
	\rho_1 u_{tt}-\kappa_1(u_x-v)_x-\gamma_1(u_x-v)_{xt}=0\quad{\rm in} \quad (0,l)\times (0,\infty),\\
	\label{TEq1.2}
	\rho_2v_{tt}-b_1v_{xx}-\kappa_1(u_x-v)-\gamma_1(u_x-v)_t+\gamma_2(-\partial_{xx})^\alpha    v_t=0\quad{\rm in} \quad (0,l)\times (0,\infty).
\end{eqnarray}
subject to boundary conditions 
\begin{equation}
	\label{EEq1.3}
	\quad u(0,t)=u(l,t)=v_x(0,t)=v_x(l,t)=0,\quad\quad t>0,
\end{equation}
and initial data
\begin{equation}
	\label{TEq1.4}
	u(x,0)=u_0(x),\; u_t(x,0)=u_1(x),\; v(x,0)=v_0(x),\; v_t(x,0)=v_1(x),\quad \quad x\in(0,l).
\end{equation}
The total energy $\mathfrak{E}_1\colon \mathbb{R}^+_0\to\mathbb{R}^+$ is given by    
\begin{equation*}\label{TEnergia01}
\mathfrak{E}_1(t)=\frac{1}{2}\bigg[ \rho_1 \|u_t\|^2+\rho_2\|v_t\|^2+ \kappa_1\|u_x-v\|^2+b_1\|A_*^{1/2}v\|^2 \bigg ](t), \quad t\geq 0.\\
\end{equation*}
Consequently,
\begin{equation*}
\dfrac{d}{dt}\mathfrak{E}_1(t)=-\gamma_1\|(u_x-v)_t\|^2 -\gamma_2\|A_*^{\alpha/2} v_t\|^2  \leq 0.
\end{equation*}
Consider 
\begin{equation*}
 \mathcal{H}_1:= D(A^\frac{1}{2})\times D(A^0) \times D(A_*^\frac{1}{2})\times D(A_*^0)
\end{equation*}
with inner product
\begin{equation*}
	\dual{ U_1}{U_2}_{\mathcal{H}}  :=  \rho_1\dual{\tilde{u_1}}{\tilde{u_2}}+\rho_2\dual{\tilde{v_1}}{\tilde{v_2}}+\kappa_1\dual{u_{1x}-v_1}{u_{2x}-v_2}
	+b_1\dual{A_{*}^{1/2}v_{1}}{A_{*}^{1/2}v_{2}},
\end{equation*}
for $U_i=(u_i,\tilde{u_i},v_i, \tilde{v_i})^T\in \mathcal{H}_1$,  $i=1,2$ and induced norm
\begin{equation}\label{TNORM}
	\|U\|_{\mathcal{H}_1}^2:=\rho_1 \|\tilde{u}\|^2+\rho_2\|\tilde{v}\|^2+ \kappa_1\|u_x-v\|^2+b_1\|A_*^{1/2}v\|^2.
\end{equation}
Condition \eqref{RBEixoIm} of Theorem \ref{LiuZAnalyticity} is proved analogously.  To prove condition \eqref{Analyticity}, it suffices to show that
\begin{equation*}
|\lambda|\|U\|_{\mathcal{H}_1}^2:=\rho_1|\lambda| \|\tilde{u}\|^2+\rho_2|\lambda|\|\tilde{v}\|^2+ \kappa_1|\lambda|\|u_x-v\|^2+b_1|\lambda|\|A_*^{1/2}v\|^2\leq C\|F\|_{\mathcal{H}_1}\|U\|_{\mathcal{H}_1}.
\end{equation*}
Acting as by the proof of items $(i)$ and $(iii)$ of Lemma  \ref{Lemma9} and item $(i)$ of Lemma \ref{Lemma10}, the desired result is easily obtained. In this way, we conclude that the Timoshenko System \eqref{TEq1.1}--\eqref{TEq1.4}, is also analytic for $\alpha\in[0,1]$.
\section{Conclusion}
We choose in this manuscript to divide the study of stability and regularity into sections 4 and 5 respectively, in order to illustrate that the theory of semigroups for linear operators can be applied to the study of both topics. It is known that analyticity of the semigroup $S(t)$ associated with the system \eqref{Eq1.1}--\eqref{Eq1.9}, implies the exponential decay of $S(t)$, as can be observed from the estimate \eqref{Analyticity} (  for $|\lambda|\geq 1$ if $|\lambda|\|U\|_\mathcal{H}\leq C_\delta\|F\|_\mathcal{H}$ then $\|U\|_\mathcal{H}\leq C_\delta\|F\|_\mathcal{H}$).

We impose the boundary conditions given by \eqref{Eq1.5}--\eqref{Eq1.6}. However, if we change these boundary conditions for only Dirichlet boundary conditions, given by
\begin{eqnarray}
	\label{Eq.1.5D}
	\quad u(0,t)=u(l,t)=v(0,t)=v(l,t)=0,\quad\quad t>0,\\
	\label{Eq.1.6D}
	y(0,t)=y(l,t)=z(0,t)=z(l,t)=0,\quad\quad t>0,
\end{eqnarray}
it will not be necessary to introduce the self-adjoint operators powers of the operator $A_*$, instead of these operators it will be enough to use the self-adjoint operators powers of the operator $-\Delta:=A\colon D(A)\cap D(A^0)\to D(A^0)$, with which the same analytical results obtained previously will be used.

{\bf Advances over previous works:}
\begin{itemize}
\item{Universality:} Analyticity does not depend on special restrictions, covering the entire parameter space.
\item{Robustness:} Stability and analyticity hold for weak and strong dissipative regimes alike.
\item{Physical Relevance:} Guarantees mathematical robustness for DWCNT models under realistic experimental conditions at the nanoscale.
\end{itemize}

Thus, we can conclude that this research represents a qualitative leap in the mathematical modeling of double-walled carbon nanotubes and Timoshenko Systems. Unlike previous results that required
strong dissipation or restricted parameter regions, this research establishes exponential stability and complete analyticity throughout the $[0,1]^2$ region. This robustness opens
perspectives for optimal control, numerical simulation, and extensions to thermoelastic
and retarded systems.


\end{document}